%% file: 0main.tex
\documentclass{article}
\usepackage[hyphens,spaces,obeyspaces]{url}
\usepackage{color}
\usepackage[square,sort,comma,numbers]{natbib}
\usepackage{amsmath,amsthm}
\usepackage{amsfonts,amssymb}
\usepackage{mathrsfs}
\usepackage{hyperref}
\usepackage{fullpage}
\usepackage[all,2cell,cmtip]{xy}
\usepackage{breakurl}
\UseTwocells
\usepackage{proof}

\usepackage{tikz}
\usetikzlibrary{arrows,decorations,automata,petri,backgrounds,circuits,circuits.ee.IEC,shapes,fit,matrix}

\tikzstyle{simple}=[-,line width=2.000]
\tikzstyle{arrow}=[-,postaction={decorate},decoration={markings,mark=at position .5 with {\arrow{>}}},line width=1.100]
\pgfdeclarelayer{edgelayer}
\pgfdeclarelayer{nodelayer}
\pgfsetlayers{edgelayer,nodelayer,main}

\tikzstyle{none}=[inner sep=0pt]

\definecolor{lblue}{rgb}{0,250,255}
\tikzstyle{species}=[circle,fill=yellow,draw=black,scale=1.15]
\tikzstyle{transition}=[rectangle,fill=lblue,draw=black,scale=1.15]
\tikzstyle{inarrow}=[->, >=stealth, shorten >=.03cm,line width=1.5]
\tikzstyle{empty}=[circle,fill=none, draw=none]
\tikzstyle{inputdot}=[circle,fill=purple,draw=purple, scale=.25]
\tikzstyle{inputarrow}=[->,draw=purple, shorten >=.05cm]
\tikzstyle{simple}=[-,draw=purple,line width=1.000]

\tikzstyle{place}=[circle,thick,draw=blue!75,fill=blue!20,minimum size=6mm]
\tikzstyle{red place}=[place,draw=red!75,fill=red!20]
\tikzstyle{transition}=[rectangle,thick,draw=black!75,
  			  fill=black!20,minimum size=4mm]

\usepackage{tikz-cd}

\usepackage{mathtools}

\usepackage{natbib}

\newcommand{\define}[1]{{\bf \boldmath{#1}}}

\mathchardef\mhyphen="2D
\newcommand{\cat}[1]{\textup{\textsf{#1}}}
\newcommand{\Set}{\mathsf{Set}}

\newcommand{\Net}[1]{{#1}\mhyphen \mathsf{Net}}
\newcommand{\qCat}[1]{{#1}\mhyphen \mathsf{Cat}}
\newcommand{\Grph}{\mathsf{Grph}}
\newcommand{\Petri}{\mathsf{Petri}}
\newcommand{\PreNet}{\mathsf{PreNet}}
\newcommand{\CMC}{\mathsf{CMC}}
\newcommand{\Mod}{\mathsf{Mod}}

\newcommand{\SSMC}{\mathsf{SMC}}
\newcommand{\SMC}{\mathsf{MC}}
\newcommand{\QNet}[5]{  \xymatrix{#1 \ar@<+.5ex>[r]^{#3} \ar@<-.5ex>[r]_{#4} & #5 #2}  }
\newcommand{\law}[1]{\mathsf{#1}}

\newcommand{\maps}{\colon}
\newcommand{\Cat}{\cat{Cat}}

\newcommand{\Mnd}{\mathsf{Mnd}}
\newcommand{\Law}{\mathsf{Law}}
\newcommand{\N}{\mathbb{N}}
\newcommand{\Z}{\mathbb{Z}}

\newcommand{\CMon}{\mathsf{CMon}}
\newcommand{\Mon}{\mathsf{Mon}}

\newcommand{\CAT}{\mathsf{CAT}}
\newcommand{\Span}{\mathsf{Span}}

\newcommand{\A}{^\bullet\mathrm{A}}
\newcommand{\B}{^{\bullet}\mathrm{B}}

\newcommand{\backb}{\mathrm{B}^{\bullet}}

\newcommand{\backa}{\mathrm{A}^{\bullet}}

\definecolor{purple(x11)}{rgb}{0.8, 0.9, 0.9}

\newcommand{\Ob}{\mathrm{Ob}}
\newcommand{\Mor}{\mathrm{Mor}}
\newcommand{\Q}{\mathsf{Q}}

\newcommand{\Hom}{\mathrm{Hom}}

\theoremstyle{plain}
\newtheorem{thm}{Theorem}[section]
\newtheorem{prop}[thm]{Proposition}
\newtheorem{lem}[thm]{Lemma}

\theoremstyle{definition}
\newtheorem{defn}[thm]{Definition}
\newtheorem{expl}[thm]{Example}
\newtheorem{rmk}[thm]{Remark}
\title{Petri Nets Based on Lawvere Theories}

\author{Jade Master}

\date{}

\begin{document}
	\maketitle
	\begin{abstract}
		We give a definition of $\law{Q}$-net, a generalization of Petri nets based on a Lawvere theory $\law{Q}$, for which many existing variants of Petri nets are a special case. This definition is functorial with respect to change in Lawvere theory, and we exploit this to explore the relationships between different kinds of $\law{Q}$-nets. To justify our definition of $\law{Q}$-net, we construct a family of adjunctions for each Lawvere theory explicating the way in which $\law{Q}$-nets present free models of $\law{Q}$ in $\Cat$. This gives a functorial description of the operational semantics for an arbitrary category of $\law{Q}$-nets. We show how this can be used to construct the semantics for Petri nets, pre-nets, integer nets, and elementary net systems.
	\end{abstract}

\section{Introduction}
Following the introduction of Petri nets in Carl Petri's 1962 thesis \citep{PetriThesis}, there has been an explosion of work on Petri nets. The bibliography hosted by \textsl{Petri Nets World} has over 8500 citations \citep{Petri}. These papers include many variations of Petri nets which change both the structure of the nets and their semantics. In this work we help organize these definitions by putting some of the more popular variants under a common framework.

Petri nets can be thought of as commutative monoidal graphs: graphs whose edges have elements of a free commutative monoid as their source and target. In this paper we generalize this to graphs which are based in some other algebraic gadget --- as long as that gadget comes from a Lawvere theory.	Petri nets are given by a pair of functions from a set of transitions to the free commutative monoid on a set of places. For a Lawvere theory $\law{Q}$, a $\law{Q}$-net is a pair of functions from a set of transitions to a free model of $\law{Q}$ in $\Set$. 
	
Many instances of this generalization are already studied. In 2000  Bruni, Meseguer, Montanari, and Sassone introduced \emph{pre-nets}; a type of Petri net that has a free non-commutative monoid on its places \citep{functorialsemantics}. The semantics of these showcase the \emph{individual token philosophy}, which distinguishes between identical tokens and keeps track of causality within sequences of processes. In 2018, Herold and Genovese introduced \emph{integer nets}, a type of Petri net which has a free abelian group of places \citep{genovese} and are similar to \emph{lending nets} introduced in \citep{lending}. These are useful for modeling credit in propositional contract logic \citep{PCL}. The last example of $\law{Q}$-nets that we consider are \emph{elementary net systems} \citep{elementary}. These are Petri nets that can have a maximum of one token in each place. The above three examples give categories 
$\PreNet$, $\Net{\law{\Z}}$, and $\Net{\law{SLAT}}$
of pre-nets, integer nets, and elementary net systems respectively. These are given by setting $\law{Q}$ equal to $\law{MON}$ the Lawvere theory of monoids, $\law{ABGRP}$ the Lawvere theory of abelian groups, and $\law{SLAT}$ the Lawvere theory of semi-lattices in the definition of $\Net{\law{Q}}$.

To elegantly illustrate the power of Petri nets, Messeguer and Montanari show that they present free monoidal categories \citep{monoids}. The objects in these monoidal categories are given by the markings of our Petri net and the morphisms represent all possible firing sequences of the transitions in sequence and in parallel. In general, a description of the operational semantics of Petri nets consists of an adjunction between a category of Petri nets and a category of strictly commutative, strictly associative monoidal categories (see Definition \ref{CMCdef}). There are many ways to do this. In \citep{open} the authors construct an adjunction from $\Petri$ into a particular subcategory of $\CMC$, the category of commutative monoidal categories. Here we analyze and alter this adjunction to get an adjunction
\[
\begin{tikzcd}
\Petri \ar[r,bend left,"F"] \arrow[r, phantom, "\bot"]& \CMC \ar[l,bend left,"U"] 
\end{tikzcd}
\]
To justify our definition of $\law{Q}$-nets, we take a similar tack. We show that every $\law{Q}$-net presents a free model of the Lawvere theory $\law{Q}$ in $\Cat$ by constructing an adjunction
\[\begin{tikzcd}  \Net{\law{Q}}  \ar[r,bend left,"F_\Q"] \arrow[r, phantom, "\bot",pos=.7] & \ar[l,bend left, "U_\Q"] \Mod(\Q,\Cat)\end{tikzcd}\]
where $\Mod(\law{Q}, \Cat)$ is the category of models of $\law{Q}$ in $\Cat$. To turn a $\law{Q}$-net into a free model of $\law{Q}$ in $\Cat$, there are two steps which must be completed:
\begin{itemize}
\item the transitions of the $\law{Q}$-net must be freely closed under the operations of $\law{Q}$, and
\item the transitions must be freely turned into a category by freely adding identities and composites.
\end{itemize}
Thus the adjunction $F_{\law{Q}} \dashv U_{\law{Q}}$ is constructed as the composite of two smaller adjunctions corresponding to these steps. This adjunction defines an operational semantics for $\law{Q}$-nets. The objects and morphisms in the free $\law{Q}$-category on a $\law{Q}$-net represent all the possible firing sequences which can be built using composition and the operations and axioms of the Lawvere theory $\law{Q}$.

An outline of this paper is as follows:
\begin{itemize}

\item In Section \ref{defin} we review definitions for Petri nets and their categorical semantics.

\item In Section \ref{QNet} we define $\law{Q}$-nets, and we extend this definition to a functor from the category of Lawvere theories to $\CAT$. This allows us to explore various functors between different kinds of $\law{Q}$-nets. We also show that the category $\Net{\law{Q}}$ is complete and cocomplete.

\item In Section \ref{CMC} we construct a semantics functor for Petri nets as a two part composite. This serves as a blueprint for the more general construction.

\item In Section \ref{gen} we prove the main theorem of this paper. For every Lawvere theory $\law{Q}$, there is an adjunction  \[\begin{tikzcd} \Net{\law{Q}}\ar[r,bend left,"F_\Q"] \ar[r,phantom,"\bot"] & \ar[l,bend left,"U_\Q"] \qCat{\law{Q}} \end{tikzcd}\] which gives the operational semantics of $\law{Q}$-nets.

\item In Section \ref{applications} we show how our main theorem can give categorical descriptions of existing forms of semantics for various types of $\law{Q}$-nets. We show how to build the semantics of Petri nets, pre-nets, and integer nets using the individual and collective token philosophies. We also construct an operational semantics functor for elementary net systems.

\item In Appendix \ref{appendix} we give a brief introduction to Lawvere theories.
\end{itemize}

\input{definitions.tex}
\input{Q-nets.tex}
\input{freecmcs}
\input{generalizedsemantics.tex}

 \input{applications.tex}

\subsection*{Acknowledgements}
I would like to thank my advisor John Baez for his constant help and guidance. I would like to thank Christina Vasilokopoulou in particular for the slick proof of Proposition \ref{complete}. I would like to thank Clemens Berger and Mike Schulman for helping me understand how morphisms of Lawvere theories turn into morphisms of monads. I would also like to thank Mike Shulman and Todd Trimble who found a mistake in an earlier draft and helped me fix it in a way which improved the paper. I thank Joe Moeller, Daniel Cicala, Kenny Courser, and Joshua Meyers for their feedback and support. Lastly, I would like to thank my friends and family for their love and support, I would not be where I am without them.

\appendix
\section{Lawvere Theories}
\label{appendix}
Introduced by Lawvere in his landmark thesis \citep{Lawvere}, Lawvere theories are a general framework for reasoning about algebraic structures \citep{ttt,buckley}.

\begin{defn}
A Lawvere theory $\law{Q}$ is a small category with finite products such that every object is isomorphic to the iterated finite product $x^n = x \times \ldots \times x$ for a \define{generic object} $x$ and natural number $n$. Equivalently, Lawvere theories can be thought of as categories whose objects are given by natural numbers $n \in \N$ and with cartesian product given by $+$. The morphisms in a Lawvere theory are called \define{operations}.
\end{defn}
\noindent The idea is that a Lawvere theory represents the platonic embodiment of an algebraic gadget.

\begin{expl}\label{mon}
A canonical example is the Lawvere theory $\law{MON}$ of monoids. Like all Lawvere theories, the objects of $\law{MON}$ are given by natural numbers. In addition $\law{MON}$ contains the morphisms 
\[ m \maps 2 \to 1 \text{ and } e \maps 0 \to 1\]
For a monoid $M$, this represents the multiplication map
\[ \boldsymbol{\cdot} \maps M \times M \to M \]
and the map
\[ e \maps \{*\} \to M\]
which picks out the identity element of $M$. 
These maps are required to satisfy the associative law
\[
\begin{tikzcd}
3 \ar[r, "\mathrm{id} \times m"] \ar[d, "m \times \mathrm{id}",swap] & 2 \ar[d, "m"]\\
2 \ar[r, "m",swap] & 1
\end{tikzcd}
\]
and the unital laws for monoids.
\[
\begin{tikzcd}
1 \ar[r, "\mathrm{id} \times e"] \ar[dr,"\mathrm{id}",swap]& 2\ar[d, "m"] & 1 \ar[l,"e \times \mathrm{id}",swap] \ar[dl,"\mathrm{id}"] \\
 & 1 &
\end{tikzcd}
\]
$\law{MON}$ also contains all composites, tensor products, and maps necessary to make $n$ into the product $x^n$ induced by the maps $m$ and $e$.
\end{expl}
\noindent Like all good things, Lawvere theories form a category. 
\begin{defn}
 Let $\cat{Law}$ be the category where objects are Lawvere theories and morphisms are product preserving functors.
\end{defn}
Note that because morphisms of Lawvere theories preserve products, they must send the generic object of their source to the generic object of their target. Therefore to specify a morphism of Lawvere theories, it suffices to make an assignment of the morphisms which are not part of the product structure.

Let $\law{Q}$ be a Lawvere theory and $C$ a category with finite products. We can impose the axioms and operations of $\law{Q}$ onto an object in $C$ via a product preserving functor $F \maps \law{Q} \to  C$. The image $F(1)$ of the generating object $1$ gives the underlying object of $F$ and for an operation $o \maps n \to k$ in $\law{Q}$, $F(o) \maps F(x)^n \to F(x)^k$ gives a specific instance of the algebraic operation represented by $o$. There is a natural way to make a category of these functors.
\begin{defn}\label{models}
Let $\law{Q}$ be a Lawvere theory and $C$ a category with finite products. Then there is a category $\Mod(\law{Q},C)$ where 
	    \begin{itemize}
			\item objects are product preserving functors $F\maps \law{Q} \to C$ and,
			\item morphisms are natural transformations between these functors.
		\end{itemize}
When $\Mod(\law{Q})$ is written without the second argument, it is assumed to be $\Set$. We will refer to objects in $\Mod(\law{Q})$ as $\law{Q}$\define{-models} and morphisms in $\Mod(\law{Q})$ as $\law{Q}$\define{-model homomorphisms}. When $C=\Cat$, we will refer to these objects as $\Q$\define{-categories}.
\end{defn}
\noindent When the category of models is $\Set$ then there is a forgetful functor 
\[ R_{\law{Q}} \maps \Mod(\law{Q}) \to \Set \] 
which sends a product preserving functor $F \maps \law{Q} \to \Set$ to  image on the generating object $F(1)$ and a natural transformation to  component on the object $1$.
A classical result says that $R_{\law{Q}}$ \emph{always} has a left adjoint
\[ L_{\law{Q}} \maps \Set \to \Mod(\law{Q})\]
which for a set $X$, $L_{\law{Q}} X$ is referred to as the \define{free model of }$\law{Q}$\define{ on }$X$. In fact, this construction extends to fully faithful functor 
\[ \Law \to \Mnd\]
which sends a Lawvere theory $\law{Q}$ to the monad $R_{\law{Q}} \circ L_{\law{Q}} \maps \Set \to  \Set$ and where $\Mnd$ is the category of monads on $\Set$ \citep{linton}. For a Lawvere theory $\law{Q}$ we will denote the monad it induces via this functor by $M_{\law{Q}} \maps \Set \to \Set$.

For $\law{Q}= \law{MON}$, $\Mod(\law{MON}, \Set)$ is equivalent to the category $\cat{Mon}$ of monoids and monoid homomorphisms. In this case the functor $R_{\law{MON}} \maps \cat{Mon} \to \Set$ turns monoids and monoid homomorphisms into their underlying sets and functions. $R_{\law{MON}}$ has a left adjoint 
\[L_{\law{MON}} \maps \Set \to \Mon\]
which sends a set $X$ to the free monoid $L_{\law{MON}} X$. For a function $f \maps X \to Y$, $L_{\law{MON}} f$ is the unique multiplication preserving extension of $f$ to $L_\law{MON} X$.

\bibliographystyle{plainnat}

\end{document}

%% file: definitions.tex
\section{Petri Nets and Their Executions}\label{defin}
\begin{defn}\label{N}
Let $L \maps \Set \to \CMon$ be the free commutative monoid functor, that is, the left adjoint of the functor $R \maps \CMon \to \Set$ that sends commutative monoids to their underlying sets and monoid homomorphisms to their underlying functions. Let 
\[ \N \maps \Set \to \Set \] be the \define{free commutative monoid monad} given by the composite $R\circ L$.
\end{defn}
\noindent For any set $X$, $\N[X]$ is the set of formal finite linear combinations of elements of $X$ with natural number coefficients.   The unit of $\N$ is given by the natural inclusion of $X$ into $\N[X]$, and for any function $f \maps X \to Y$, $\N[f] \maps \N[X] \to \N[Y]$ is the unique monoid homomorphism that extends $f$.   

\begin{defn}\label{PetriNet}
We define a \define{Petri net} to be a pair of functions of the following form:
\[\xymatrix{ T \ar@<-.5ex>[r]_-t \ar@<.5ex>[r]^-s & \N[S]. } \]
We call $T$ the set of \define{transitions}, $S$ the set of \define{places}, $s$ the \define{source} function, and $t$ the \define{target} function. 
	
\end{defn}

\begin{defn}\label{PetriMorphism}
A \define{Petri net morphism} from the Petri net 
$\xymatrix{ T \ar@<-.5ex>[r]_-t \ar@<.5ex>[r]^-s & \N[S] }$ to
the Petri net $\xymatrix{ T' \ar@<-.5ex>[r]_-{t'} \ar@<.5ex>[r]^-{s'} & \N[S']}$ 
is a pair of functions $(f \maps T \to T', g \maps S \to S')$ such that the diagrams 

	\[
	\xymatrix{ 
		T \ar[d]_f  \ar[r]^-{s} & \N[S] \ar[d]^-{\N[g]} \\	
		T' \ar[r]_-{s'} & \N[S'] 
	}
	\qquad
	\xymatrix{ 
		T \ar[d]_f  \ar[r]^-{t} & \N[S] \ar[d]^-{\N[g]} \\	
		T' \ar[r]_-{t'} & \N[S'] . 
	}
	\]
	commute.
\end{defn}

\begin{defn}
Let $\Petri$ be the category of Petri nets and Petri net morphisms, with composition 
defined by 
\[  (f, g) \circ (f',g') = (f \circ f' , g \circ g')  .\]
\end{defn}
\noindent Our definition of Petri net morphism differs from the earlier definition used by Degano--Meseguer--Montanari \cite{DMM} and Sassone \cite{SassoneStrong,SassoneAxiom,functorialsemantics}. The difference is that our definition requires that the homomorphism between free commutative monoids come from a function between the sets of places whereas the above references allow arbitrary commutative monoid homomorphisms. This difference of definition is present in our definition of the category of $\law{Q}$-nets as well. With this change, the categories $\Petri$ and $\Net{\law{Q}}$ become complete and cocomplete as shown in Proposition \ref{complete}.

Petri nets have a natural semantics which is described by \emph{the token game}. This is a game where each place of a Petri net is equipped with a natural number of tokens. Players are then allowed to shuffle the tokens around in proportion given by their source and target. The token game is formalized by the notions of marking and firing. 
\begin{defn}
A \define{marking} of a Petri net $P = \begin{tikzcd}T \ar[r,shift left=.5ex,"s"] \ar[r,shift right=.5ex,"t",swap] & \N[S] \end{tikzcd}$ is an element $m \in \N[S]$, or equivalently, a function $m \maps S \to \N$ which is zero on all but a finite number of elements. A \define{firing} of $P$ is a tuple $(\tau,x,y)$, where $\tau$ is a transition and $x$ and $y$ are markings of $P$ with $x - s(\tau) \geq 0$ and $x-s(\tau) + t(\tau) = y$.
\end{defn}
Firings can be chained together in sequence: for a firing $(\tau,x,y)$ and a firing $(\sigma,y,z)$ we can define their composite as a tuple $(\sigma \circ \tau,x,z)$ where $\circ$ is a formal symbol. Firings can also be performed in parallel: for two firings $(\tau,x,y)$ and $(\tau',x',y')$ there is a parallelization $(\tau +\tau', x+x',y+y')$. This suggests that firings of a Petri net have the structure of a monoidal category. Meseguer and Montanari were the first to notice this and show how Petri nets can be turned into commutative monoidal categories \cite{monoids}.
\begin{defn}\label{CMCdef}
	A \define{commutative monoidal category} is a commutative monoid object internal to $\Cat$. Explicitly, a commutative monoidal category is a strict monoidal category $(C,\otimes,I)$, such that, for all objects $a$, $b$ and morphisms $f$, $g$ in $C$ 
	\[a \otimes b = b \otimes a \text{ and } f \otimes g = g \otimes f.\]
\end{defn}
 
\noindent Note that a commutative monoidal category is the same as a strict symmetric monoidal category where the symmetry isomorphisms $\sigma_{a,b} \maps a \otimes b \stackrel{\sim}{\longrightarrow} b \otimes a$ are all identity morphisms. In fact, a commutative monoidal category is precisely a category where the objects and morphisms form commutative monoids and the structure maps are commutative monoid homomorphisms. A commutative monoidal category where the morphisms represent sequences of firings of a Petri net $P$ will be referred to as a semantics for $P$. In this paper we characterize this semantics construction as an adjunction between the category of Petri nets and the following category.
\begin{defn}
	Let $\CMC$ be the category whose objects are commutative monoidal categories and whose morphisms are strict monoidal functors.
\end{defn}
Note that every monoidal functor between commutative monoidal categories is automatically a strict symmetric monoidal functor, so the adjective symmetric is not included in the above definition.

%% file: Q-nets.tex
\section{Q-Nets}\label{QNet}
Petri nets need not have a free commutative monoid of places, and this aspect can be generalized using Lawvere theories. A review of the basic definitions and properties of Lawvere theories can be found in the appendix. As in Definition \ref{models}, let $\Mod(\law{Q})$ be the category of models of $\law{Q}$ in $\Set$,
\[ \begin{tikzcd} \Set \ar[r,bend left,"L_\Q",pos=.6] \ar[r,phantom,"\bot",pos=.6] & \ar[l,bend left,"R_\Q",pos=.4] \Mod(\law{Q}) \end{tikzcd}\]
be the adjunction generating free models of $\law{Q}$ and let $M_\law{Q}$ be the composite $R_\law{Q} \circ L_\law{Q}$.
\begin{defn}\label{QNetd}	
	Let $\Net{\law{Q}}$ be the category where
\begin{itemize}
	    \item objects are $\law{Q}$\define{-nets}, i.e. pairs of functions of the form
	    \[ \begin{tikzcd}
	    T \ar[r, shift left=.5ex, "s"] \ar[r, shift right=.5ex,"t",swap] & M_{\law{Q}} S
	    \end{tikzcd} \]
		\item a morphism from the $\law{Q}$-net 
$\xymatrix{ T \ar@<-.5ex>[r]_-t \ar@<.5ex>[r]^-s & M_\law{Q} S }$ to
the $\law{Q}$-net $\xymatrix{ T' \ar@<-.5ex>[r]_-{t'} \ar@<.5ex>[r]^-{s'} & M_\law{Q} S'}$ 
is a pair of functions $(f \maps T \to T', g \maps S \to S')$ such that the following diagrams commute:
	\[
	\xymatrix{ 
		T \ar[d]_f  \ar[r]^-{s} & M_\law{Q} S \ar[d]^-{M_\law{Q} g} \\	
		T' \ar[r]_-{s'} & M_\law{Q} S' 
	}
	\qquad
	\xymatrix{ 
		T \ar[d]_f  \ar[r]^-{t} & M_\law{Q}S \ar[d]^-{M_\law{Q} g} \\	
		T' \ar[r]_-{t'} & M_\law{Q} S' . 
	}
	\]
	\end{itemize}
\end{defn}
	This definition is a construction. Let $M_\law{Q}$ be as before and let $M_{\law{R}} \maps \Set \to \Set$ be corresponding monad induced by a Lawvere theory $\law{R}$. Every morphism of Lawvere theories $f \maps \law{Q} \to \law{R}$ induces a functor
	\[f_* \maps \Mod(\law{R}) \to  \Mod(\law{Q})\]
	which composes every model of $\law{R}$ with $f$. A left adjoint
	\[f^* \maps \Mod(\law{Q}) \to \Mod(\law{R}) \] is given by the left Kan extension of each model along $f$ \cite{ttt,buckley}. Now, we have the following commutative diagram of functors
	
	\[
	\xymatrix{\Mod(\law{Q}) \ar[dr]_{R_\law{Q}}  & &\ar[ll]_{f_*} \Mod(\law{R}) \ar[dl]^{R_{\law{R}}} \\
	& 
	 \Set  & }
	\]
	all of which have left adjoints. Given this set of assumptions, there is a morphism of monads $M^f$ given by
	\[M^f = R_\law{Q} \eta L_\law{Q} \maps M_\law{Q} \Rightarrow  M_{\law{R}} \]
	where $\eta$ is the unit of the adjunction $f^* \dashv f_*$. This can either be verified directly, or by using the adjoint triangle theorem \cite{triangles}. In what follows we will use this morphism of monads to translate between different types of generalized $\law{Q}$-nets.

\begin{defn}\label{netf}
	    Let 
	    \[ \Net{(-)}  \maps \Law \to \Cat \]
	    be the functor which sends a Lawvere theory $\law{Q}$ to the category $\Net{\law{Q}}$ and sends a morphism $f\maps \law{Q} \to \law{R}$ of Lawvere theories to the functor  $\Net{f} \maps \Net{\law{Q}} \to \Net{\law{R}}$ which sends a $\law{Q}$-net 
            \[
            \begin{tikzcd}
            T \ar[r, shift left=.5ex, "s"] \ar[r,shift right=.5ex,"t",swap] & M_{\law{Q}} S
            \end{tikzcd}
            \]
to the $\law{R}$-net
            \[ 
            \begin{tikzcd}
            T \ar[r, shift left=.5ex, "s"] \ar[r,shift right=.5ex,"t",swap] & M_{\law{Q}} S \ar[r,"M_S^{f}"] & M_{\law{R}} S
            \end{tikzcd} 
            \] 

\noindent For a morphism of $\law{Q}$-nets $(g\maps T \to T',h \maps S \to S')$, $\Net{f} (g,h)$ is $(g,h)$. This is well-defined because of the naturality of $M^f$.
\end{defn}
\noindent Varying the Lawvere theory $\Q$ gives many known types of Petri nets.
\begin{expl}
Setting $\Q$ equal to $\law{CMON}$, the Lawvere theory for commutative monoids, we obtain the category of Petri nets.
\end{expl}

\begin{defn}\label{prenetdef}
Let $(-)^* \maps \Set \to \Set$ denote the monad that the Lawvere theory $\law{MON}$ induces via the correspondence in \cite{linton}. For a set $X$, $X^*$ is given by the underlying set of the free monoid on $X$.
 A \define{pre-net} is a pair of functions of the form
 \[
 \begin{tikzcd}
 T \ar[r, shift left=.5ex, "s"] \ar[r, shift right=.5ex, "t",swap] & S^*
 \end{tikzcd}
 \]
 A morphism of pre-nets from a pre-net $(s,t \maps T \to S^*)$ to a pre-net $(s',t'\maps T' \to S'^*)$ is a pair of functions $(f \maps T \to T', g\maps S \to S')$ which preserves the source and target as in Definition \ref{PetriMorphism}. This defines a category $\PreNet$.
\end{defn}
\begin{expl} \label{prenet}
If we take $\law{Q}= \law{MON}$, the Lawvere theory of monoids, we get the category $\PreNet$.
\end{expl} \noindent A description of $\law{MON}$ can be found in the Appendix. $\PreNet$ has the same objects as the category introduced in \emph{Functorial Models for Petri Nets} \cite{functorialsemantics} but the morphisms are restricted as in Definition \ref{PetriMorphism}. Pre-nets are the same as tensor schemes introduced by Joyal and Street in "The Geometry of Tensor Calculus I" \cite{scheme}. The authors define a notion of free category on a tensor scheme and Bruni, Meseguer, Montanari, and Sassone construct an adjunction between pre-nets and a subcategory of the category of strict symmetric monoidal categories $\SSMC$ \cite{functorialsemantics}.

In Section \ref{prenetapp} we construct a closely related adjunction 
	
	\[\begin{tikzcd}
	\PreNet \ar[r,bend left,"Z"] \ar[r,phantom,"\bot"] & \ar[l,bend left, "K"] \SSMC
	\end{tikzcd} \]which does not require the restriction to a subcategory of $\SSMC$. Pre-nets are useful because after forming an appropriate quotient, the category $Z(P)$ for a pre-net $P$ is equivalent the category of \textit{strongly concatenable processes} which can be performed on the net. This equivalence is important for realizing the individual token philosophy \cite{functorialsemantics}. The individual token philosophy, as opposed to the collective token philosophy, gives identities to the individual tokens and keeps track of the causality in the executions of a Petri net.
\begin{expl}
In 2013 Bartoletti, Cimoli, and Pinna introduced lending Petri nets \cite{lending}. These are Petri nets where arcs can have a negative multiplicity and tokens can be borrowed in order to fire a transition. Lending nets are also equipped with a partial labeling of the places and transitions so they can be composed and are required to have no transitions which can be fired spontaneously. In 2018 Genovese and Herold introduced integer nets \cite{genovese}. Let $\law{ABGRP}$ be the Lawvere theory of abelian groups. This Lawvere theory contains three generating operations
\[e \maps 0 \to 1 \text{,  } i \maps 1 \to 1 \text{, and } m \maps 2 \to 1\]
representing the identies, inverses, and multiplication of an abelian group. These generating morphisms are required to satisfy the axioms of an abelian group; associativity, commutativity, and the existence of inverses and an identity.
The category of integer nets, modulo a change in the definition of morphisms, can be obtained by taking $\law{Q} = \law{ABGRP}$ in the definition of $\Net{\law{Q}}$.
\begin{defn}\label{integernets}
    Let $\Z \maps \Set \to \Set$ be the free abelian group monad which for a set $X$ generates the free abelian group $\Z [X]$ on the set $X$. Note that $\Z$ is the monad induced by the Lawvere theory $\law{ABGRP}$ via the correspondence in \cite{linton}. An \define{integer net} is a pair of functions of the form
    \[\xymatrix{ T \ar@<-.5ex>[r]_-t \ar@<.5ex>[r]^-s & \Z[S]. } \]
    A \define{morphism of integer nets} is a pair $(f \maps T \to T', g \maps S\to S')$ which makes the diagrams analogous to the definition of Petri net morphism (Definition \ref{PetriMorphism}) commute. 
    Let $\Net{\Z}$ be the category where objects are integer nets and morphism are morphisms of integer nets.
\end{defn}
\noindent Integer nets are useful for modeling the concepts of credit and borrowing. There is a correspondence between lending Petri nets and propositional contract logic; a form of logic useful for ensuring that complex networks of contracts are honored \cite{lending}. Genovese and Herold constructed a categorical semantics for integer nets \cite{genovese}. In Section \ref{applications} we will construct a variation of this semantics which uses the general framework developed in this paper.
\end{expl}

\begin{expl}
Elementary net systems, introduced by Rozenberg and Thiagarajan in 1986, are are Petri nets with a maximum of one edge between a given place and transition \cite{elementary}.

\begin{defn}
 An \define{elementary net system} is a pair of functions 
 \[
 \begin{tikzcd}
 T  \ar[r, shift left=.5ex, "s"] \ar[r, shift right=.5ex, "t", swap] & 2^S
 \end{tikzcd}
 \] 
 where $2^S$ denotes the power set of $S$.
\end{defn}

\noindent Elementary net systems can be obtained from our general formalism. Let $\law{SLAT}$ be the Lawvere theory for semi-lattices, i.e. commutative idempotent monoids. This Lawvere theory contains morphisms
\[ m \maps 2 \to 1 \text{ and } e \maps 0 \to 1\]
as in $\law{MON}$ the theory of monoids. Also similar to $\law{MON}$, $\law{SLAT}$ is quotiented by the associativity and unitality axioms given in Example \ref{mon}. In addition, $\law{SLAT}$ has the following axioms representing commutativity and idempotence
\[
\begin{tikzcd}
2 \ar[rr, "\sigma"] \ar[dr, "m",swap] & & 2 \ar[dl, "m"] \\
& 1 &
\end{tikzcd}
\quad \quad 
\begin{tikzcd}
1\ar[rr, "\Delta"] \ar[dr, "id",swap] & & 2 \ar[dl, "m"] \\
& 1 &
\end{tikzcd}
\]
where $\sigma \maps 2 \to 2$ is the braiding of the cartesian product and $\Delta \maps 1 \to 2$ is the diagonal. For models in $\Set$, The first diagram says that you can multiply two elements in either order and the get the same thing. The second diagram says that if you multiply an element by itself you get itself. As in Definition \ref{models}, $\law{SLAT}$ corresponds to a monad on $\Set$. It is well known that this monad is the covariant power set monad

\[
2^{(-)} \maps \Set \to \Set
\]
which sends a set $X$ to its power set and a function to the mapping which sends subsets of $X$ to their image. This motivates the following:
\begin{defn}
 Let $\Net{\law{SLAT}}$ be the category of elementary net systems obtained as in Definition \ref{QNet} for $\law{Q}= \law{SLAT}$.
\end{defn} 
\end{expl}
The functorial nature of Definition \ref{QNet} can be exploited to generate functors between different categories of $\law{Q}$-nets. 
There is the diagram in $\cat{Law}$
\begin{center}
\begin{tikzcd}
\law{SLAT}  &            &\\
\law{CMON} \ar[r,"b"] \ar[u,"a"] & \law{ABGRP} &\\
\law{MON} \ar[u,"c"] \ar[r,"d",swap] & \law{GRP} \ar[u,"e",swap]
\end{tikzcd}
\end{center}
where all the morphisms send their generating operations to their counterparts in the Lawvere theory of their codomain. These target Lawvere theories either have extra axioms or operations making the above functors not necessarily full or faithful:
\begin{itemize}
    \item $a$ sends the morphism $m \circ \Delta$ in $\law{CMON}$ to $id_1 \maps 1 \to 1$ in $\law{SLAT}$ to impose the idempotent law. All other other generating components are sent to their natural counterparts.
    \item $b$ and $d$ send every object and morphism to its natural analog. However, $\law{ABGRP}$ and $\law{GRP}$ have an extra operation $i \maps 1 \to 1$ representing inverses. This makes the functors $b$ and $d$ faithful but not full.
    \item $c$ and $e$ add the commutativity law; they send both the multiplication $m \maps 2 \to 1$ and the composite $\sigma \circ m\maps 2 \to 2$ of the braiding $\sigma \maps 2 \to 2$ to the multiplication map $m \maps 2 \to 2$ in the target Lawvere theory. This makes $c$ and $e$ not faithful.
\end{itemize}

\noindent Definition \ref{netf} can be used to give a network between different flavors of $\law{Q}$-nets. By applying $\Net{(-)}$ to the above diagram we get the diagram of categories

\begin{center}
\begin{tikzcd}
\Net{\law{SLAT} } &            &\\
\Petri \ar[r,"\Net{b}"] \ar[u,"\Net{a}"] & \Net{\Z} &\\
\PreNet \ar[u,"\Net{c}"] \ar[r,"\Net{d}",swap] & \Net{\law{GRP}} \ar[u,"\Net{e}",swap]
\end{tikzcd}
\end{center}
The functors in this diagram are described as follows:
\[\Net{c} \maps \PreNet \to \Petri \]
 is often called \textit{abelianization} because it sends a pre-net to the Petri net which forgets about the ordering on the input and output of each transition. The authors of \cite{functorialsemantics} use $\Net{c}$ to explore the relationship between pre-nets and Petri nets. The functor $\Net{e}\maps \Net{\law{GRP}} \to \Net{\Z}$ gives the analogous relationship for integer nets.
\[\Net{b} \maps \Petri \to \Net{\Z}\]
is the functor which does not change the source and target of a given place. The only difference is that the markings of a $\Z$-net coming from a Petri net are thought of as elements of a free abelian group rather than a free abelian monoid. $\Net{d}$ is the analogous functor for pre-nets. 
\[ \Net{a} \maps \Petri \to \Net{\law{SLAT}} \]
is the functor which sends a Petri net to the $\law{SLAT}$-net which forgets about the multiplicity of the edges between a given source and transition. 

Before moving on to the semantics of $\law{Q}$-nets, we discuss a property of the category $\Net{\law{Q}}$.

\begin{prop}\label{complete}
$\Net{\law{Q}}$ is cocomplete.
\end{prop}
\begin{proof}
We can construct $\Net{\law{Q}}$ as the comma category. 
\[  \Set \downarrow (\times \circ \Delta \circ M_\law{Q} ) \]
where $M_\law{Q} \maps \Set \to \Set$ is the monad corresponding to the Lawvere theory $\law{Q}$, $\Delta \maps \Set \to \Set \times \Set$ is the diagonal, and $\times \maps \Set \times \Set \to \Set$ is the cartesian product in $\Set$. An object in this category is a map
\[T \to M_\law{Q} S \times M_\law{Q} S\]
\noindent which corresponds to a pair of maps $s,t \maps T \to M_\law{Q} S$ which become the source and target maps of a $\law{Q}$-net. Morphisms in this comma category are commutative squares

\[ 
\xymatrix{ T\ar[r] \ar[d]_{f} & M_\law{Q} S \times M_\law{Q} S \ar[d]^{M_{\law{Q}} g \times M_{\law{Q}} g } \\
            T' \ar[r] & M_\law{Q} S' \times M_\law{Q} S'}
\]
giving a map of $\law{Q}$-nets $(f \maps T \to T',g \maps S \to S')$. The commutativity of the above square ensures that this map of $\law{Q}$-nets is well-defined. 

Theorem 3, Section 5.2 of \textit{Computational Category Theory} \cite{comp} says that given $S \maps A \to C$ and $T \maps B \to C$ then the comma category $(S \downarrow T)$ is cocomplete if
\begin{itemize}
\item $S$ is cocontinuous and,
\item $A$ and $B$ are cocomplete,
\end{itemize}
Because $\Set$ is cocomplete and the identity functor $1_{\Set} \maps \Set \to \Set$ preserves all colimits, we have that $\Set \downarrow (\times \circ \Delta \circ M_\law{Q} )$ is cocomplete. Because $\Net{\law{Q}}$ is equivalent to this category, it is cocomplete as well.
\end{proof}

%% file: freecmcs.tex
\section{Generating Free Commutative Monoidal Categories From Petri Nets}\label{CMC}
In this section we examine in detail the motivating example for the main result of this paper, an adjunction generating the semantics of $\law{Q}$-nets for every Lawvere theory $\law{Q}$. This result can feel abstract on its own and the example of Petri nets provides invaluable intuition. A confident reader may skip this section, as it is not strictly necessary for the rest of this paper.

The operational semantics for Petri nets will take the form of an adjunction
\[
\begin{tikzcd}
\Petri \ar[r,bend left,"F"] \ar[r,phantom,"\bot"] & \CMC \ar[l,bend left,"U"] 
\end{tikzcd}
\]
For a given Petri net $P$, this adjunction will be constructed in two steps: first the transitions of $P$ will be freely closed under a commutative monoidal sum and then freely closed under composition. This will take the form of factoring the adjunction into the composite
\[
\begin{tikzcd}
\Petri \ar[r,bend left,"\A"] \ar[r,phantom,"\bot",pos=.6] & \Grph(\CMon) \ar[l,bend left,"\backa",pos=.45] \ar[r,phantom,"\bot",pos=.4] \ar[r,bend left,"\B"] & \CMC \ar[l,bend left,"\backb",pos=.55].
\end{tikzcd}
\]
Here a left adjoint is indicated by a bullet on the left and a right adjoint is indicated by a bullet on the right. $\Grph(\CMon)$ is the category of graphs internal to $\CMon$.
\begin{defn}
A \define{commutative monoidal graph} is a graph
\[
\begin{tikzcd}
E \ar[r,shift left=.5ex,"s"] \ar[r, shift right=.5ex,"t",swap] & V
\end{tikzcd}
\]
where $E$ and $V$ are commutative monoids and $s$ and $t$ are commutative monoid homomorphisms. A morphism of commutative monoidal graphs is a tuple of commutative monoid homomorphisms $(f \maps E \to E', g \maps V \to V')$ making the diagrams
	\[
	\xymatrix{ 
		E \ar[d]_f  \ar[r]^-{s} & V \ar[d]^-{g} \\	
		E' \ar[r]_-{s'} & V' 
	}
	\qquad
	\xymatrix{ 
		E \ar[d]_f  \ar[r]^-{t} & V \ar[d]^-{g} \\	
		E' \ar[r]_-{t'} & V' 
	}
	\]commute. This defines a category $\Grph(\CMon)$ where objects are commutative monoidal graphs and morphisms are as above. In short, $\Grph(\CMon)$ is the category of graphs internal to $\CMon$.
\end{defn}
We will now define these adjunctions but omit the proofs that they are indeed well-defined adjunctions, as this follows from the more general results of Section \ref{gen}.
The left adjoint $\A \maps \Petri \to \Grph(\CMon)$ is defined as follows:
\begin{defn}\label{A}
Let 
\[ \A \maps \Petri \to \Grph(\CMon)\]
be the functor which sends a  Petri net 
\[
\begin{tikzcd}
P = T \ar[r,shift left=.5ex,"s"] \ar[r, shift right=.5ex,"t",swap] & \N[S]
\end{tikzcd}
\]
to the commutative monoidal graph
\[
\begin{tikzcd}
\A P = LT \ar[r,shift left=.5ex,"\phi^{-1}(s)"] \ar[r,shift right=.5ex,"\phi^{-1}(t)",swap] & LS
\end{tikzcd}
\]
where $L$ is the left adjoint of the adjunction in Definition \ref{N} and $\phi \maps \Hom (LT, LS) \xrightarrow{\sim} \Hom(T,RLS)$ is the natural isomorphism of that adjunction. $\A$ sends a morphism of Petri nets
\[(f \maps T \to T', g \maps S \to S')\]
to the morphism of commutative monoidal graphs given by 
\[(Lf \maps LT \to LT',Lg \maps LS \to LS') \]
\end{defn}
In words, $\A$ freely generates a commutative monoidal structure on the transitions of a Petri net and $\A$ uniquely extends each component of a Petri net morphism to a commutative monoid homomorphism. 
The right adjoint of this functor is non-trivial:
\begin{defn}
Let 
\[\backa \maps \Grph(\CMon) \to  \Petri \]
be the functor which sends a commutative monoidal graph
\[\begin{tikzcd}
Q=E \ar[r,shift left=.5ex,"s"] \ar[r, shift right=.5ex,"t",swap] & V
\end{tikzcd} \]
to the Petri net 
\[ 
\begin{tikzcd}
\backa Q = \bar{E} \ar[r,shift left=.5ex,"\bar{s}"] \ar[r, shift right=.5ex,"\bar{t}",swap] & \N[RV]
\end{tikzcd}
\]
$\bar{E}$ is defined as 
\[\bar{E}= \{(e,x,y) \in RE \times \N[RV] \times \N[RV] | R\epsilon_V(x)=s(e) \text{ and } R \epsilon_V(y)=t(e) \} \]
where $\epsilon_V$ is the counit of the adjunction $L \dashv R$. $\bar{s}$ and $\bar{t}$ are given by the projection of $\bar{E}$ onto its first and second coordinates respectively. $\backa$ sends a morphism of commutative monoidal graphs
\[(f\maps E \to E',g\maps V \to V')\]
to the morphism of Petri nets 
\[(h \maps \bar{E} \to \bar{E'}, Rg \maps RV \to RV') \]
where $h$ is the function which makes the assignment
\[(e,x,y) \mapsto (\phi(f)(e),\N[Rg](x),\N[Rg](y) ) \]
\end{defn}
\begin{rmk}\label{failure}
Petri nets must have a free commutative monoid of places, so it is necessary to regard $RV$ as the set of places for $\backa Q$ rather than having $V$ be the commutative monoid of places itself. The reader at this point may guess a simpler formula for the right adjoint $\backa$ which keeps the $RE$ as the set of transitions and uses the unit of $\N$ to construct the source and target maps. Unfortunately this construction is doomed to fail. For a commutative monoidal graph $\begin{tikzcd}E \ar[r,shift left=.5ex,"s"] \ar[r,shift right=.5ex,"t",swap] & V \end{tikzcd}$, suppose that the right adjoint $\A$ sends this graph to the Petri net
\[
\begin{tikzcd}
RE \ar[r,shift left=.5ex,"\eta \circ Rs"] \ar[r,shift right=.5ex,"\eta\circ Rt",swap]& \N[RV].
\end{tikzcd}
\]
A problem arises because this process unnaturally chunks the source and target of each transition. To see this consider the commutative monoidal graph 
\[
\begin{tikzcd}
Q= \N[\tau] \ar[r,shift left=.5ex] \ar[r,shift right=.5ex] & \N[\{a,b,c\}] \cong \N^3
\end{tikzcd}
\]The edge $\tau$ in $Q$ can be depicted as
\[
\begin{tikzpicture}[node distance=1.3cm,>=stealth',bend angle=45,auto]
\node [place] (a) [label=left:$a$] at (0,0) {};
\node [place] (b) [label=left:$b$] at (0,-1) {};
\node [place] (c) [label=above:$c$] at (2,-.5) {};
\node [transition] (t) [label=$\tau$] at (1,-.5) {}
    edge [pre]          (a)
    edge [pre]          (b)
    edge [post]         (c);
\end{tikzpicture}
\]
With the above (faulty) description, $\backa Q$ is given by
\[
\begin{tikzcd}
\N[\tau] \ar[r,shift left=.5ex] \ar[r, shift right=.5ex] & \N[\N^3]
\end{tikzcd}
\]
To avoid confusion, we denote the outer sum in $\N[\N^3]$ by $\times$ and the sum in $\N^3$ by $+$. Then, the faulty $\backa$ would turn $\tau$ into the transition
\[
\begin{tikzpicture}[node distance=1.3cm,>=stealth',bend angle=45,auto]
\node [place] (ab) [label=above:$a+b$] at (-1,0) {};
\node [place] (c) [label=above:$c$] at (1,0) {};
\node [transition] (t) [label=$\tau$] at (0,0) {}
    edge [pre]          (ab)
    edge [post]         (c);
\end{tikzpicture}
\]
To find a counit for this adjunction we seek a morphism
\[ \A \backa Q \to Q  \]
A morphism of this sort is defined by its assignment on generators. A natural choice of morphism sends the places $a+b$ to the sum of the places $a$ and $b$ using the counit of $L \dashv R$. However, then the assignment of $\tau \mapsto \tau$ does not respect the source of $\tau$ and is therefore not a morphism of commutative monoidal graphs. The problem is that we want the source of $\tau$ in $\A \backa Q$ to be $a \times b$ and not $a+b$. To fix this we force the source of $\tau$ to be $a\times b$ by upgrading $\tau$ to the tuple $(\tau, a \times b, c)$ in $\N[\tau] \times \N[\N^3] \times \N[\N^3]$. Now, the natural choice for the counit which sends $a \times b$ to $a +b$ respects the source of $\tau$.
\end{rmk}

The next part of the semantics adjunction for Petri nets freely generates the structure of a category on a given commutative monoidal graph. In Section \ref{gen} this is accomplished by rephrasing this construction in terms of free monoids. Here we provide an explicit description in the case of Petri nets. 

\begin{defn}
Let 
\[ \backb \maps \CMC \to \Grph(\CMon)\]
be the forgetful functor which sends a commutative monoidal category to its underlying commutative monoidal graph and a strict monoidal functor to its underlying morphism of commutative monoidal graphs. Then $\backb$ has a left adjoint 
\[ \B \maps \Grph(\CMon) \to \CMC\]
which sends a commutative monoidal graph
\[ 
\begin{tikzcd}
Q= E \ar[r,shift left=.5ex,"s"] \ar[r,shift right=.5ex,"t",swap] & V
\end{tikzcd}
\]
to the commutative monoidal category $\B Q$ with objects given by $V$ and morphisms generated inductively by the rules:
\begin{itemize}
    \item for every edge $e \in E$ a morphism $e \maps s(e) \to t(e)$,
    \item for every pair of morphisms $e \maps x \to y$ and $d \maps y \to z$, a morphism $d \circ e \maps x \to z$,
    \item for every object $v \in V$ a morphism $1_v \maps v \to v$
\end{itemize}
This defines an evident composition operation $\circ$ on $\B Q$.
There is also a sum on the $\B Q$ defined using the sum of $V$ on objects. If $e$ and $e'$ are edges of $Q$ then the morphisms $e \maps x \to y$ and $e' \maps x' \to y'$ already have a sum given by
\[e+e' \maps x +x' \to y+y' \]
 The morphisms of $\A Q$ are quotiented by the axioms:
\begin{itemize}
    \item for all tuples of morphisms $(f \maps a \to b, g \maps b \to c, h \maps c \to d)$ 
    \[(f \circ g) \circ h = f \circ (g \circ h)\]
    \item for all morphisms $f \maps x \to y$ 
    \[1_y \circ f = f = f \circ 1_x \] 
    \item We require that composition is a commutative monoid homomorphism. For tuples of morphisms $(e \maps x \to y, d \maps y \to z, e' \maps x' \to y', d' \maps y' \to z')$ we can form their sum and composite in two different ways. We quotient the morphisms of $\B Q$ so that these are equal, i.e.
    \[(d \circ e) + (d' \circ e') = (d +d') \circ (e+e')\]
    \item We require that the assignment of identities is a commutative monoid homomorphism. For objects $x$ and $x'$ in $V$ we set
    \[1_{x+x'}=1_x +1_x'\]
\end{itemize}
\end{defn}

%% file: generalizedsemantics.tex
\section{Semantics Functors for Generalized Nets}\label{gen}
In this section we construct semantics categories for $\law{Q}$-nets; categories whose morphisms represent possible sequences of firings which can be performed using a given $\Q$-net. 
Let $\law{Q}$ be a Lawvere theory and let
\[\begin{tikzcd}
\Set \ar[r,bend left=45,"L"]\ar[r,phantom,"\bot",pos=.7] & \ar[l,bend left=45,"R"] \Mod(\law{Q}, \Set)
\end{tikzcd}
\] be the adjunction it induces on $\Set$. In this section we will use this adjunction to construct an adjunction
\[ 
\begin{tikzcd}
  \Net{\law{Q}} \ar[r,bend left,"F_{\Q}"] \ar[r,phantom,"\bot",pos=.6] & \ar[l,bend left,"U_{\Q}"] \Mod(\Q,\Cat)
\end{tikzcd}
\] which is analogous to the adjunction in Section \ref{CMC} and where $\Mod(\Q,\Cat)$ is the category of models of $\Q$ in $\Cat$. This adjunction factors as

\[ 
\begin{tikzcd}
\Net{\Q} \ar[r,bend left,"\A_\Q"] \ar[r,phantom,"\bot",pos=.6] & \ar[l,bend left, "{\backa}_\Q"] \Grph(\Mod(\Q)) \ar[r,bend left, "\B_\Q"]\ar[r,phantom,"\bot"] & \ar[l,bend left,"\backb_\Q"] \Mod(\Q,\Cat)
\end{tikzcd}
\]
where $\A_\Q$ freely generates a model of $\Q$ on the transitions of a given $\Q$-Net and $\B_\Q$ freely generates the structure of a category on a given $\Q$-graph.

These adjunctions are heavily motivated by the case when $\Q= \law{CMON}$ as this gives Petri nets. The main result of this paper is as follows:

\begin{thm}\label{big}
There is an adjunction
\[ 
\begin{tikzcd}
  \Net{\law{Q}} \ar[r,bend left,"F_{\Q}"] \ar[r,phantom,"\bot",pos=.6]& \ar[l,bend left,"U_{\Q}"] \Mod(\Q,\Cat).
\end{tikzcd}
\]
\end{thm}\noindent The left adjoint can be described using inference rules. 
Let $P$ be the $\law{Q}$-net
\[
\begin{tikzcd}
P=T \ar[r, shift left=1, "s"] \ar[r, shift right=1, "t",swap] & M_\law{Q} S
\end{tikzcd}
\]
The objects of $F_{\law{Q}}$ are given by $L_{\law{Q}} S$. That is, for every morphism $o \maps n \to m$ in $\law{Q}$ and every tuple of places $x_1, x_2, \ldots, x_n$ there is an object $\mathbf{o}(x_1, x_2, x_3, \ldots, x_n)$. For an equation of morphisms in $\law{Q}$
\begin{center}
\begin{tikzcd}
& n \ar[dl, "f",swap] \ar[dr,"h"] \\
m \ar[rr,"g",swap] & &k
\end{tikzcd}
\end{center}
the objects generated by each path must be equal. This means that there are $k$ equations of objects
\[ \mathbf{g_j} (\mathbf{f_1} (x_{1}, x_{2}, \ldots, x_{n}), \mathbf{f_2} (x_{1},x_{2}, \ldots, x_{n} ) , \ldots, \mathbf{f_m} (x_{1},x_{2},\ldots, x_{n}) ) = \mathbf{h_j} (x_{1}, x_2, \ldots, x_n)  \]
where the unlabeled index runs over the components of $f$ and the index $j$ runs over the components of $g$ and $h$. The morphisms of $F_Q P$ are generated inductively by the rules
\[
\infer{\tau \maps s(\tau) \to t(\tau) \in \Mor\, F_\law{Q} P}{\tau \in T} \quad \quad \infer{1_x \maps x \to x \in \Mor\, F_\law{Q} P}{x \in \Ob\, F_\law{Q} P} \quad \quad \infer{g \circ f \maps x \to z \in \Mor \, F_\law{Q} P}{f \maps x \to y \text{ and } g \maps y \to z \in \Mor\, F_\law{Q} P }
\]
\[
\infer{\mathbf{o} (f_1,f_2, \ldots, f_n) \maps \mathbf{o} (x_1,x_2, \ldots,x_n) \to \mathbf{o} (y_1,y_2, \ldots, y_n)}{o \maps n \to 1 \in \Mor\, \law{Q} \text{ and } f_1\maps x_1 \to y_1,f_2 \maps x_2 \to y_2, \ldots f_n \maps x_n \to y_n \in \Mor \, F_\law{Q} P}
\]
and is quotiented to satisfy the following:
\begin{itemize}
\item The morphisms must satisfy the same equations that the objects satisfy. That is, for an equation of morphisms in $\Q$, the objects generated by each path must again be equal. 

\item $\Mor\, F_\law{Q} P$ is quotiented to satisfy the axioms of a category including the associative and unital laws
\[ (f \circ g) \circ h = f \circ (g \circ h) \text{  and  } 1_y \circ f = f = f \circ 1_x \]
for all morphisms $f$, $g$ and $h$ in $\Mor \, F_{\law{Q}} P$.
\item $\Mor\, F_\law{Q} P$ is quotiented so that the structure maps of a category (source, target, identity and composition) are $\law{Q}$-model homomorphisms. 
\end{itemize}

For a morphism of $\Q$-nets, $(f,g) \maps P \to P'$, the $\Q$-functor
\[ F_{\Q} (f,g) \maps F_{\Q} P \to F_{\Q} P' \]
is the unique extension of $f$ and $g$ which respects composition, unitality, and the operations of $\Q$.
The proof will require several lemmas.

The first step is to construct an adjunction which freely closes the transitions of a $\Q$-net under the operations of $\Q$. In this adjunction we will write the monad $M_\law{Q}$ as $R L$ and make use of the natural isomorphism 
\[\phi \maps \hom(L X, Y) \xrightarrow{\sim} \hom(X, R Y)\]
for all sets $X$ and objects $Y$ in $\Mod(\law{Q})$.
\begin{defn}
Let\[ \A_\law{Q} \maps \Net{\law{Q}} \to \Grph(\Mod(\Q))\] be the functor which makes the assignment
\[
\begin{tikzcd}
T \ar[d,"f",swap] \ar[r, shift left=.5ex,"s"] \ar[r,shift right=.5ex,"t",swap] & RLS \ar[d,"RLg"{name=L}] & & LT \ar[d,"Lf"{name=R},swap] \ar[r, shift left=.5ex,"\phi^{-1}(s)"] \ar[r,shift right=.5ex,"\phi^{-1}(t)",swap] & LS \ar[d,"Lg"] \\
T' \ar[r, shift left=.5ex,"s'"] \ar[r,shift right=.5ex,"t'",swap] & RLS'& & LT' \ar[r, shift left=.5ex,"\phi^{-1}(s')"] \ar[r,shift right=.5ex,"\phi^{-1}(t')",swap] & LS' \ar[mapsto,from=L,to=R,shorten <=2ex,shorten >=2ex]
\end{tikzcd}
\]
on objects and morphisms.
\end{defn}

\begin{lem}
    $\A_\law{Q}$ is well-defined.
\end{lem}
The next few proofs will make heavy use of the naturality equations for $\phi$ and its inverse:
\[  \phi(a \circ b \circ Lc) = R a \circ \phi(b) \circ c\] and 
\[\phi^{-1} (Ra \circ b \circ c) = a \circ \phi^{-1}(b) \circ Lc.\]
\begin{proof}
First we show that $Le$ commutes with the source of $\A_\law{Q} P$. This follows from the chain of equalities:
\begin{align*}
 & \phi^{-1} (s') \circ Lf \\
 &= \phi^{-1} (s' \circ f) \\
 &= \phi^{-1} (RLg \circ s)\\
 &= Lg \circ \phi^{-1} (s) 
 \end{align*}
 A similar equation holds for the target maps.
\end{proof}
Let $G$ be the $\Q$-graph 
\[\begin{tikzcd}
E \ar[r,shift left=.5ex,"s"] \ar[r,shift right=.5ex,"t",swap] & V.
\end{tikzcd} \]Because $V$ is not a free model of $\Q$, there is no obvious forgetful way to turn this into a $\Q$-net.  A first guess for the $\Q$-net $\backa_\Q (G)$ might be the $\Q$-net
\[\begin{tikzcd}RE \ar[r,shift left=.5ex,"\eta_{RV} \circ s"] \ar[r,shift right=.5ex,"\eta_{RV} \circ t",swap] & M_\Q (RV)\end{tikzcd} \]
where $\eta_{RV}$ is the unit of the monad $M_\Q$ applied to the set $RV$. However, as explained in Remark \ref{failure}, this fails to be a right adjoint. An alternative approach was suggested by Mike Shulman in the comments of an nCaf\'e blog post \cite{nCafe}. This solution was inspired by the construction of the free category on a tensor scheme introduced in \emph{The Geometry of Tensor Calculus I} \cite{scheme}. Instead of using $RE$ as the set of transitions, we use
    \[ \bar{E} = \{ (e,x,y) \in RE \times M_\Q RV \times M_\Q RV \, | \, s(e) = R\epsilon_{V} (x)\text{ and } t(e) = R \epsilon_{V}(y) \}\]
where $\epsilon_{V}$ is the $V$ component of the counit for $M_\Q$. Here and in what follows we are using $s$ to denote $Rs$ and $t$ to denote $Rt$ for notational simplicity. The source and target maps of the resulting $\Q$-net are given by the projections of $\bar{E}$ onto its second and third coordinates. The set $\bar{E}$ can be described formally using pullbacks.
\begin{defn}\label{backa}
Let
\[ \backa_{\Q} \maps \Grph(\Q)\to \Net{\Q} \] 
be the functor which makes the assignment on objects and morphisms
\[
\begin{tikzcd}
E \ar[d,"f",swap] \ar[r, shift left=.5ex,"s"] \ar[r,shift right=.5ex, "t",swap] & V  \ar[d,"g"{name=L}] & \bar{E}\ar[d,"\bar{f}"{name=R},swap]\ar[mapsto,shorten <=1ex,shorten >=1ex,from=L,to=R] \ar[r, shift left=.5ex,"\bar{s}"] \ar[r, shift right=.5ex,"\bar{t}",swap] & M_\Q RV \ar[d,"M_\Q R g"]\\
E  \ar[r, shift left=.5ex,"s"] \ar[r,shift right=.5ex, "t",swap] & V   & \bar{E'}  \ar[r, shift left=.5ex,"\bar{s'}"] \ar[r, shift right=.5ex,"\bar{t'}",swap] & M_\Q RV' & 
\end{tikzcd}
\]
where
\begin{itemize}
    \item $\bar{E}$ is the pullback of sets
    \[
    \begin{tikzcd}[column sep={4em,between origins}]
 & \ar[dl,"i",swap] \bar{E} \ar[dr,"j"] & \\
    RE \ar[dr,"{(s,t)}",swap] & & M_\Q RV \times M_\Q RV \ar[dl,"R\epsilon_V \times R\epsilon_V"] \\
    & RV\times RV &
    \end{tikzcd}
    \]
   where $(s,t)$ denotes the pairing of $s$ and $t$, and $\epsilon_{RV}\times \epsilon_{RV}$ denotes the cartesian product of the counits.
    \item $\bar{s} \maps \bar{E} \to M_\Q RV$ is the composite
    \[
    \begin{tikzcd}
    \bar{E} \ar[r,"j"] & M_\Q RV \times M_\Q RV \ar[r,"\pi_1"] & M_Q RV
    \end{tikzcd}
    \]
    and $\bar{t} \maps \bar{E} \to M_\Q RV$ is the composite
      \[
    \begin{tikzcd}
    \bar{E} \ar[r,"j"] & M_\Q RV \times M_\Q RV \ar[r,"\pi_2"] & M_Q RV
    \end{tikzcd}
    \]
    that is the maps which send an element $(e,x,y)$ of $\bar{E}$ to its second and third coordinates.

   \item $\bar{f} \maps \bar{E} \to \bar{E'}$ is induced by the universal property of $\bar{E}$ as shown below
 \[
    \begin{tikzcd}
 & \ar[dl,"i",swap] \bar{E} \ar[ddddd,dashrightarrow,bend right=90,looseness=2,"\bar{f}",swap] \ar[dr,"j"] & \\
    RE \ar[ddd,"Rf",swap]\ar[dr,"{(s,t)}",swap] & & M_\Q RV \times M_\Q RV \ar[dl,"R\epsilon_V \times R\epsilon_V"] \ar[ddd,"M_\Q Rg \times M_\Q Rg"] \\
    & RV\times RV \ar[d,"Rg \times Rg"] &\\
    & RV' \times RV' & \\
    RE' \ar[ur, "{(s',t')}"] & & M_\Q RV' \times M_\Q RV' \ar[ul, "R\epsilon_{RV'} \times R\epsilon_{RV'}",swap,pos=.7] \\
     & \bar{E'} \ar[ur,"j'",swap] \ar[ul,"i'"]&
    \end{tikzcd}
    \]
    More simply, $\bar{f}$ makes the assignment
    \[ (e,x,y) \mapsto (Rf(e),M_\Q Rg (x), M_\Q Rg (y) )\]
\end{itemize}
\end{defn}

\begin{lem}
    $\backa_\law{Q}$ is well-defined.
\end{lem}

\begin{proof}
We must show that $(\bar{f}, Rg)$ is a well-defined morphism of $\Q$-nets.
$\bar{f}$ and $Rg$ commute with the source and target maps. Indeed, using the elementary descriptions we get that
\begin{align*}
    \bar{s'} \circ \bar{f} (\tau,x,y) &= \bar{s} ( Rf(\tau), M_\Q Rg(x), M_\Q Rg (y) ) \\
    &= M_\Q Rg (x) \\
    &= M_\Q Rg ( \bar{s} (\tau,x,y) )
\end{align*}
A similar equation holds for the target maps. $(\bar{f}, Rg)$ commutes with the identity maps:
\begin{align*}
    \bar{f} \circ \bar{e} (x) &= \bar{f} (Re (x), \eta_{RV} (x), \eta_{RV} (x) ) \\
    &=(Rf \circ Re(x), M_\Q Rg \circ \eta_{RV} (x), M_\Q Rg \circ \eta_{RV} (x) ) \\
    &= (Re' \circ Rg (x), M_\Q Rg \circ \eta_{RV} (x), M_\Q Rg \circ eta_{RV} (x) )\\
    &= (Re' \circ Rg (x),\eta_{RV'} \circ Rg (x),\eta_{RV'} \circ Rg (x) )\\
    &= \bar{e'} \circ Rg (x)
\end{align*}
where the last two steps follow from naturality of $\eta$ and $(f,g)$ being a morphism of $\Q$-graphs.
\end{proof}

\begin{lem}
$\backa_\Q$ is a right adjoint to $\A_\Q$.
\end{lem}

\begin{proof}
Let $P$ be the $\Q$-net
\[
\begin{tikzcd}
T \ar[r,shift left=.5ex,"s"] \ar[r,shift right=.5ex,"t",swap] & RLS
\end{tikzcd}
\]
and $Q$ be the $\Q$-graph
\[
\begin{tikzcd}
E \ar[r,shift left=.5ex,"s'"] \ar[r, shift right=.5ex,"t'",swap] & V
\end{tikzcd}
\]
We define a natural isomorphism
\[ \Phi \maps  \Hom(\B P, Q) \xrightarrow{\sim} \Hom(P, \backb_\Q Q)\]
by the rule
\[
\begin{tikzcd}
LT \ar[d,"f",swap] \ar[r, shift left=.5ex,"\phi^{-1} (s)"] \ar[r, shift right=.5ex,"\phi^{-1}(t)",swap] & LS \ar[d,"g"{name=L}]  & & T \ar[d,"h"{name=R},swap] \ar[r, shift left=.5ex,"s"] \ar[r,shift right=.5ex,"t",swap] & RLS \ar[d,"RL \phi(g)"] \ar[mapsto,from=L,to=R,shorten <=3ex, shorten >=3ex]\\
E \ar[r,shift left=.5ex,"s'"] \ar[r, shift right=.5ex,"t'",swap]  & V & & \bar{E} \ar[r,shift left=.5ex,"\bar{s'}"] \ar[r,shift right=.5ex,"\bar{t'}",swap] & RLRV
\end{tikzcd}
\]
$h$ is defined by the universal property induced by $\bar{E}$ and the diagram
\[
\begin{tikzcd}[column sep={6em,between origins}]
& \ar[ddl] \bar{E} \ar[ddrr] & & \\
&                          & & \\
RE \ar[dr,"{(Rs',Rt')}",swap] &\ar[l,"\phi(f)",swap] T \ar[uu,dashed,"h"] \ar[rr,"{(RL \phi(g) \circ s, RL \phi(g) \circ t)}"] & & RLRV \times RLRV \ar[dll,"R\epsilon_V \times R\epsilon_V"] \\
& RV \times RV & &
\end{tikzcd}
\]
This diagram is well defined because $T$ is a competitor to the pullback $\bar{E}$ i.e. it makes the lowest triangle commute. Checking this amounts to showing that the bottom square commutes and this can be verified componentwise:
\begin{align*}
R\epsilon_V \circ RL \phi(g) \circ s &=R(\epsilon_V \circ L \phi(g)) \circ s \\
&= R( \phi^{-1} (1_RV) \circ L \phi(g) ) \circ s \\
&= R( \phi^{-1} ( 1_{RV} \circ \phi(g) ) ) \circ s \\
& = R( \phi^{-1} (\phi(g) ) ) \circ s \\
&= Rg \circ s \\
& = Rg \circ \phi( \phi^{-1} (s) ) \\
&= \phi (g \circ \phi^{-1} (s) ) \\
&= \phi(s' \circ f) \\
&= Rs' \circ \phi(f)
\end{align*}
and similar equations hold for the target maps. Therefore, $h$ is well defined.  Explicitly $h$ is the map which makes the assignment on transitions in $T$
\[ \tau \mapsto ( \phi(f)(\tau), RL \phi(g) \circ s(\tau), RL \phi(g) \circ t(\tau) ). \]
 $(h, \phi(g) )$ is a well-defined morphism of $\Q$-graphs by construction. The source and target functions map elements to their second and third coordinates so the equation 
\[ \bar{s'} \circ h = RL \phi(g) \circ s\]
is true.

An inverse to $\Phi$,
\[\Phi^{-1} \maps \Hom ( P, \backb_{\Q} Q) \to \Hom( \B_{\Q} P , Q) \]
is defined as follows
\[
\begin{tikzcd}
T \ar[d,"h",swap] \ar[r, shift left=.5ex,"s"] \ar[r, shift right=.5ex,"t",swap] & RLS \ar[d,"RLg"{name=L}]  & & LT \ar[d,"\phi^{-1}(a)"{name=R},swap] \ar[r, shift left=.5ex,"\phi^{-1}(s)"] \ar[r,shift right=.5ex,"\phi^{-1}(t)",swap] & LS \ar[d," \phi^{-1}(g)"] \ar[mapsto,from=L,to=R,shorten <=2ex, shorten >=2ex]\\
\bar{E} \ar[r,shift left=.5ex,"\bar{s'}"] \ar[r, shift right=.5ex,"\bar{t'}",swap]  & RLRV & & E \ar[r,shift left=.5ex,"s'"] \ar[r,shift right=.5ex,"t'",swap] & V
\end{tikzcd}
\]$a$ is defined by the universal property of $\bar{E}$ and the diagram
\[
\begin{tikzcd}
& \ar[dl,"i",swap] \bar{E} \ar[dr,"j"] &  \\
RE \ar[dr,"{(Rs',Rt')}",swap] &\ar[l,dashed,"a",swap] T \ar[u,"h"] \ar[r,"b",dashed] &  RLRV \times RLRV \ar[dl,"R\epsilon_V \times R\epsilon_V"] \\
& RV \times RV & &
\end{tikzcd}
\]
To show that $(\phi^{-1}(a), \phi^{-1}(g))$ is a well defined morphism of $\Q$-graphs we perform the computation:
\begin{align*}
    s' \circ \phi^{-1}(a) &= \phi^{-1}(Rs' \circ a)\\
    &= \phi^{-1} (R \epsilon_V \circ \pi_1 \circ b) \\
    &= \phi^{-1} (R \epsilon_V \circ \bar{s'} \circ h)
\end{align*}
where $\pi_1 \maps RLRV \times RLRV \to RLRV$ is the projection and the last two steps follow from the definition of $\bar{s'}$ and commutativity of the above diagram. This can be reduced using the fact that $h$ commutes with the source and target of $P$ and $\backb_{\Q} Q$ and naturality of $\phi^{-1}$. Indeed,
\begin{align*}
    \phi^{-1} ( R \epsilon_V \circ \bar{s'} \circ h ) & = \phi^{-1} (R \epsilon_V \circ RLg \circ s) \\
    &=\phi^{-1} ( R(\epsilon_V \circ Lg) \circ s ) \\
    &= \epsilon_V \circ Lg \circ \phi^{-1} (s) \\
    &=\phi^{-1} (1_{RV}) \circ Lg \circ \phi^{-1} (s)\\
    &= \phi^{-1} (1_{RV} \circ g)\circ \phi^{-1} (s)\\
    &= \phi^{-1} (g) \circ \phi^{-1} (s) 
\end{align*}
A similar equation holds for target so this is a well-defined morphism of $\Q$-graphs. $\Phi$ is a natural isomorphism if it is a natural and a bijection in the places component and the transitions component. The places component is only an application of $\phi$ so it is both natural and a bijection. For the transition component let $D \maps C \to \Set$ be the diagram
\[
\begin{tikzcd}
RE \ar[dr,"{(Rs,Rt)}",swap]& & M_{\Q} RV \times M_{\Q}RV \ar[dl,"R\epsilon_V \times R\epsilon_V"]\\
& RV \times RV & \\
\end{tikzcd}
\]
where $C$ is the walking cospan. Let $\Delta_T \maps C \to \Set$ be the constant diagram which sends every object to $T$ and every morphism to $1_T$. Then the universal property of $\bar{E}$ can be expressed as the natural isomorphism
\[\Psi \maps \mathrm{Nat} (\Delta_T, D) \xrightarrow{\sim}  \Hom (T, \bar{E})\]
where $\mathrm{Nat}(\Delta_T, D)$ denotes the set of natural transformations from $\Delta_T$ to $D$. With this description, the transition component of $\Phi$ can be described as follows
\[
\Phi \maps \, f \,\mapsto \,\Psi(\langle \phi(f), (RL \phi(g) \circ s, RL \phi(g) \circ t)  \rangle)
\]
where the angle brackets encase the components of a natural transformation. Similarly, the transition component of $\Phi^{-1}$ can be described as
\[
\Phi^{-1} \maps \, h\, \mapsto \,\phi^{-1}(\Psi^{-1} (h)_{RE})
\]
where the subscript $RE$ indicates that we take the $RE$ component of the natural transformation. With this description, we can verify that they are inverses on the transition component. 
\begin{align*}
    f & \mapsto \Psi(\langle \phi(f), (RL \phi(g) \circ s, RL \phi(g) \circ t)  \rangle) \\
    & \mapsto \phi^{-1}(\Psi^{-1} (\Psi(\langle \phi(f), (RL \phi(g) \circ s, RL \phi(g) \circ t)  \rangle))_{RE} ) \\
    &= \phi^{-1}(\langle \phi(f), (RL \phi(g) \circ s, RL \phi(g) \circ t)  \rangle_{RE} )\\
    &= \phi^{-1}(\phi(f)) \\
    &= f
\end{align*}
and the other direction
\begin{align*}
    h & \mapsto \phi^{-1} ( \Psi^{-1}(h)_{RE} ) \\
      & \mapsto \Psi( \langle \phi ( \phi^{-1} (\Psi^{-1} (h)_{RE})), (RL \phi(g) \circ s, RL \phi(g) \circ t) \rangle ) \\
      &= \Psi \Psi^{-1} \langle h, (RL \phi(g) \circ s, RL \phi(g) \circ t \rangle_{RE} \\
      &= \langle h, (RL \phi(g) \circ s, RL \phi(g) \circ t \rangle_{RE}\\
      &=h
\end{align*}
The transition component of $\Phi$ and $\Phi^{-1}$ are natural because they are made up of components which are individually natural transformations.
\end{proof}

The next step in the proof of Theorem \ref{big}, is to construct an adjunction between $\Grph(\Mod(\Q))$ and $\Mod(\Q,\Cat)$. A general property of algebraic theories $\law{P}$ and $\Q$ is the property that models of $\law{P}$ in the category of models of $\law{Q}$ are the same as models of  $\Q$ in the category of models of $\law{P}$. In particular for a Lawvere theory $\Q$, a model of $\Q$ in $\Cat$ is the same as a category internal to $\Mod(\Q)$ and this extends to an equivalence of categories
\[\Mod(Q,\Cat) \cong \Cat(\Mod(\Q)).  \]
Therefore, the adjunction we seek is between the categories $\Grph(\Mod(\Q))$ and $\Cat(\Mod(\Q))$ i.e. a construction of free categories internal to $\Mod(\Q)$. The free category construction in this generality was first given in \cite{baues1997}. However, to build this adjunction we use a Theorem of Lack \cite{lack}. 

\begin{thm}\label{lack}{\normalfont [Lack]}
Let $(C, \otimes)$ be a monoidal category with
\begin{itemize}
    \item finite limits,
    \item countable colimits and,
    \item the functors $- \otimes A$ and $A \otimes -$ preserve reflexive coequalizers and colimits of countable chains.
\end{itemize}
Then $C$ admits a free monoid construction, that is, a left adjoint to the forgetful functor
\[ \mathrm{Mon} (C) \to C\]
which sends every monoid to its underlying object of $C$.
\end{thm}
By choosing an appropriate monoidal category we can use this construction to get free categories in $\Mod(\Q)$.
 A span of $\Q$-models of the form
\[ \begin{tikzcd}
 & \ar[dl,"s",swap] E \ar[dr,"t"] & \\
 V & & V
\end{tikzcd}
\]
is a graph in $\Mod(\Q)$. This description can be used to describe a monoidal category of graphs over $V$.

\begin{defn}
Let $\Span (V)$ be the monoidal category where
\begin{itemize}
    \item objects are given by spans $s,t \maps E \to V$ in $\Mod(\Q)$,
    \item morphisms are given by maps $f \maps E \to E'$ making the diagram
    \[
    \begin{tikzcd}
    & \ar[dl,"s",swap] E \ar[dd,"f"] \ar[dr,"t"] & \\
    V & & V\\
    &\ar[ul,"s'"] E'\ar[ur,"t'",swap]&
    \end{tikzcd}
    \]commute.
    \item monoidal product is given by chosen pullbacks. That is, for spans
    \[ \begin{tikzcd}
 & \ar[dl,"a",swap] E \ar[dr,"b"] & & & \ar[dl,"c",swap] F \ar[dr,"d"] & \\
 V & & V & V & & V
\end{tikzcd}
\]
their monoidal product is the chosen pullback 
\[
\begin{tikzcd}
& &\ar[dl] E \times_{V} F \ar[dr] & & \\
  &\ar[dl,"a",swap]E\ar[dr,"b"]&  &\ar[dl,"c",swap]F\ar[dr,"d"]& \\
V & &V & & V
\end{tikzcd}
\]
On morphisms $f \maps E \to E'$ and $g \maps F \to F'$ is the unique map
\[(f,g) \maps E \times_V F \to E' \times_V F'\] 
induced by the universal property of $E' \times_V F'$.
\end{itemize}
\end{defn}

A monoid in this monoidal category is a span $s,t \maps E \to V$ along with multiplication and unit maps 
\[ \circ \maps E \times_V E \to E \text{ and } e \maps V \to E\]
satisfying associativity and unitality. Interpreting $\circ$ as composition and $e$ as the map assigning identity morphisms, this gives a category internal to $\Mod(\Q)$. Indeed, a category with object model $V$ is exactly a monoid in the category $\Span (V)$ \cite{betti1996formal}. Therefore, if $\Span (V)$ satisfies the hypotheses of Theorem \ref{lack}, it gives a construction of free categories in $\Mod(\Q)$. The hypotheses of Theorem \ref{lack} require that the following conditions hold: \begin{itemize}
    \item $\Span (V)$ has finite limits and countable colimits. $\Mod(\Q)$ has these limits and colimits as shown in Theorem 3.4.5 of \cite{Borceux2}. The corresponding limits and colimits in $\Span (V)$ are computed on the apex of each span.
    \item The monoidal product of $\Span (V)$ preserves colimits of countable chains. This is true if pullbacks in $\Mod(\Q)$ preserve these colimits. 
    \item The monoidal product of $\Span(V)$ preserves reflexive coequalizers. This is true if pullbacks in $\Mod(\Q)$ preserve reflexive coequalizers.
\end{itemize}
Corollary 3.4.3 of \cite{Borceux2} states that finite limits commute with filtered colimits. In particular this means that pullbacks commute with colimits of countable chains and reflexive coequalizers. Therefore, Theorem \ref{lack} gives an adjunction
\[
\begin{tikzcd}
\Span (V) \ar[r,bend left, "\B_V"] & \Cat(V) \ar[l,bend left,"\backb_V"]
\end{tikzcd}
\]

However, we would like an adjunction between the category of all graphs and categories in $\Mod(\Q)$. To accomplish this, we use the Grothendieck construction \cite{Borceux2}.

\begin{defn}
Let 
\[\Span(-) \maps \Mod(\Q) \to \Cat\]
be the functor which sends an object $V$ to the category $\Span(V)$ of spans over $V$. For a morphism $f \maps V \to W$ in $\Mod(\Q)$, let 
\[\Span(f) \maps \Span(V) \to \Span(W) \]
be the functor which makes the assignment 
\[
\begin{tikzcd}
 & & & & &\ar[dl] E \ar[dddd,"k"] \ar[dr] &\\
 & \ar[dl] E \ar[dd,"k"] \ar[dr] & & & V \ar[d,"f",swap] & & V \ar[d,"f"] \\
 V & & V & \mapsto & W & & W \\
 & \ar[ul] E' \ar[ur] & & & V \ar[u,"f"] & & V \ar[u,"f",swap] \\
 & & & & &\ar[ul] E' \ar[ur] &\\
\end{tikzcd}
\]on objects and morphisms. Let
\[ \Cat(-) \maps \Mod(\Q) \to \Cat \]
be the functor which sends an object $V$ to the category of small categories internal to $\Mod(\Q)$ with object model of $\Q$ given by $V$. For a morphism $f \maps V \to W$, let
\[\Cat(f) \maps \Cat(V) \to \Cat(W) \]
be the functor which which makes the assignment 
\[
\begin{tikzcd}
\Mor\, C \ar[d,"k",swap] \ar[r, shift left=.5ex,"s"] \ar[r, shift right=.5ex,"t",swap] & V \ar[d,"\mathrm{id}"{name=L}] &  & \Mor\, C  \ar[d,"k"{name=R},swap]\ar[r, shift left=.5ex,"s"] \ar[r, shift right=.5ex,"t",swap] & V\ar[d,"\mathrm{id}"] \ar[r,"f"] & W \ar[d,"\mathrm{id}"] \ar[mapsto,from=L,to=R,shorten <=3ex, shorten >=3ex] \\
\Mor\, C' \ar[r, shift left=.5ex,"s'"] \ar[r,shift right=.5ex,"t'",swap] & V & & \Mor \, C'  \ar[r, shift left=.5ex,"s'"] \ar[r, shift right=.5ex,"t'",swap] & V \ar[r,"f",swap] & W
\end{tikzcd}
\]on the underlying graphs of objects and morphisms.
\end{defn}The adjunction derived from Theorem \ref{lack} can be reframed in this context.
\begin{prop}
The family of adjunctions 
\[
\begin{tikzcd}
\Span (V) \ar[r,bend left, "\B_V"] & \Cat(V) \ar[l,bend left,"{\backb}_V"]
\end{tikzcd}
\]
form components of natural transformations 
\[ C \maps \Span(-) \Rightarrow \Cat(-) \text{  and  } \backb \maps \Cat(-) \Rightarrow \Span(-) \]
Furthermore, $C$ and $\backb$ form an adjoint pair in the 2-category $[\Mod(\Q), \Cat]$ where
\begin{itemize}
    \item objects are functors $F \maps \Mod(\Q) \to \Cat$,
    \item morphisms are natural transformations $\alpha \maps F \Rightarrow G$ whose components $\alpha_c \maps F(c) \to G(c)$ are functors and,
    \item 2-morphisms are modifications $\gamma \maps \alpha \to \beta$. That is, for every object $c$ in $\Mod(\Q)$ a natural transformation of the type
    \[
    \begin{tikzcd}
    F(c) \ar[d,bend right=50,"\alpha_c"{name=L},swap] \ar[d, bend left=50,"\beta_c"{name=R}] \\
    G(c) \ar[Rightarrow, from=L, to=R,shorten <= 1.7ex, shorten >= 1.7ex,"\gamma_c"].
    \end{tikzcd}
    \]
\end{itemize}
\end{prop}

\begin{proof}
For naturality, it suffices to show that the squares 
\[
\begin{tikzcd}
\Span(V) \ar[d,"\Span(f)",swap] \ar[r,"\B_V"] & \Cat(V) \ar[d,"\Cat(f)"] & \Span(V) \ar[d,"\Span(f)",swap]  &\ar[l,"{\backb}_V",swap] \Cat(V) \ar[d,"\Cat(f)"]\\
\Span(W) \ar[r,"\B_W",swap] & \Cat(W) & \Span(W) & \ar[l,"{\backb}_V"]  \Cat(W)
\end{tikzcd}
\]commute. This is verified by direct computation. To show that $\B$ and $\backb$ are an adjoint pair we need the following fact: $\B$ is a left adjoint to $\backb$ in $[\Mod(\Q),\Cat]$ if and only if the components
\[
\begin{tikzcd}
\Span (V) \ar[r,bend left, "\B_V"] & \Cat(V) \ar[l,bend left,"{\backb}_V"]
\end{tikzcd}
\]
form an adjoint pair in $\Cat$. The counit-unit definition of adjunction requires that we have modifications $\epsilon \maps \B \circ \backb \to 1_{\Cat(-)}$ and $\eta \maps 1_{\Span(-)} \to \backb  \circ \B$ satisfying the snake equations. Unpacking this gives components $\epsilon_V \maps \B_V \circ \backb_V \Rightarrow 1_{\Cat(V)}$ and $\eta_V \maps \backb_V \circ \B_V \Rightarrow 1_{\Span(V)}$ satisfying the snake equations. This is equivalent to each component being an adjunction. However, Theorem \ref{lack} says that each component is an adjunction so the claim is shown.
\end{proof}

So far we have the diagram 
\[
\begin{tikzcd}
\Mod(\Q)\, \ar[r,bend left=70,"\Span(-)"{name=U}] \ar[r,bend right=70,"\Cat(-)"{name=D},swap] & \ar[Rightarrow,from=U, to=D,shorten <= 1.7ex, shorten >= 1.7ex,bend right,"\B"description] \ar[Rightarrow,from=D, to=U,shorten <= 1.7ex, shorten >= 1.7ex,bend right,"\backb"description]\Cat
\end{tikzcd}
\]of adjoint 1-cells in 
$[\Mod(\Q),\Cat]$. We apply the Grothendieck construction to this diagram to get 
\[
\begin{tikzcd}
\int \Span(-) \ar[r,bend left,"\int \B"] & \int \Cat(-) \ar[l,bend left,"\int \backb"]
\end{tikzcd}
\]The Grothendieck construction is a 2-functor $\int \maps [\Mod(\Q), \CAT]\to \CAT/\Mod(\Q)$ where $\CAT$ denotes the 2-category of large categories, functors, and natural transformations. When composed with the forgetful 2-functor $\CAT/\Mod(\Q) \to \CAT$ which remembers only the domain of each functor, we obtain the composite
\[\int \maps [\Mod(\Q),\CAT] \to  \CAT\]
which we denote as $\int$ in an abuse of notation.

A fundamental fact is that every 2-functor preserves adjunctions. Therefore the above diagram is an adjunction. Moreover, the following proposition shows that it is the adjunction we are looking for.

\begin{prop}\label{equiv}
The category $\int \Span(-)$ is equivalent to $\Grph(\Mod(\Q))$ and the category $\int \Cat(-)$ is equivalent to $\Mod(\Q,\Cat)$.
\end{prop}

\begin{proof}
$\int \Span(-)$ has
\begin{itemize}
    \item pairs $(V,V \leftarrow E \to V)$ as objects and,
    \item pairs $(f \maps V \to V', g \maps E \to E')$ such that the diagram  
    \[
    \begin{tikzcd}
    & \ar[dl] E \ar[ddd,"g"] \ar[dr] & \\
    V \ar[d,"f",swap]& & V\ar[d,"f"]\\
    V' & & V'\\
    & \ar[ur] E' \ar[ul] &
    \end{tikzcd}
    \]in $\Mod(\Q)$ commutes as morphisms.
    
\end{itemize}
An equivalence $\int \Span(-) \xrightarrow{\sim} \Grph(\Mod(\Q))$ sends $(V,V \leftarrow E \to V)$ to the graph $\begin{tikzcd}
E \ar[r, shift left=.5ex] \ar[r, shift right=.5ex,swap] & V
\end{tikzcd}$
and a morphism $(f,g)$ to the evident morphism of graphs $(f \maps E \to E', g \maps V \to V')$.

$\int \Cat(-)$ has 
\begin{itemize}
    \item pairs $(V,C)$ where $C$ is a category over $V$ as objects and,
    \item pairs $(f\maps V \to V', g \maps C \to C')$ where $g$ is an object fixing functor from $\Cat(f) (C)$ to $C'$ as morphisms.
\end{itemize}
An equivalence $\int \Cat(-) \xrightarrow{\sim} \Cat(\Mod(\Q)$ is given by sending objects $(V,C)$ to their second component and morphisms $(f,g)$ to the functor whose object component is $f$ and whose morphism component is the morphism component of $g$.
\end{proof}

We denote the compositions of $\int \B$ and $\int \backb$ with the above equivalences by $\B_{\Q}$ and $\backb_{\Q}$ respectively.

\noindent \textbf{Proof of Theorem \ref{big}.}
The composite adjunction $F_\law{Q} \dashv U_\law{Q}$ is constructed by setting $F_\law{Q} = \B_{\law{Q}} \circ \A_{\law{Q}}$ and $U_\law{Q} = \backb_\law{Q} \circ \backa_\law{Q}$. \hfill \qedsymbol\\
\smallskip

%% file: applications.tex
\section{Applications}\label{applications}
 
 Theorem \ref{big} has many applications: it can be used to help understand existing constructions of semantics for various $\law{Q}$-nets from a categorical perspective.

 \subsection{Semantics for Pre-nets}\label{prenetapp}
In \cite{functorialsemantics} the authors construct an adjunction for pre-nets which highlights the individual token semantics. In this subsection we characterize a variation of this adjunction using Theorem \ref{big}. Theorem \ref{big} gives the following adjunction for pre-nets:
\begin{prop}
Let $R_{\law{MON}} \maps \Mod(\law{MON}) \to \Set$ be the underlying set functor and let $L_{\law{MON}} \maps \Set \to \Mod(\law{MON})$ be its left adjoint. Recall that the composite $R_{\law{MON}} \circ L_{\law{MON}}$ is denoted by $(-)^*$. Let $\SMC$ be the category of strict monoidal categories and strict monoidal functors.
 Let 
 \[U_{\law{MON}} \maps \SMC \to \PreNet\] 
 be the functor which makes the assignment on objects and morphisms
 \[
\begin{tikzcd}
C \ar[d,"F"{name=L}] & \overline{\Mor\, C} \ar[d,""{name=R}] \ar[r,shift left=.5ex] \ar[r, shift right=.5ex] & \Ob\, C^* \ar[d] \ar[mapsto,from=L,to=R,shorten <=3ex,shorten >=4ex]\\
D & \overline{\Mor\, D} \ar[r,shift left=.5ex] \ar[r, shift right=.5ex] & \Ob\, D^* 
\end{tikzcd}
 \]
where $\overline{\Mor\,C}$ and the source and target maps are as defined in Definition \ref{backa}. Then, $U_{\law{MON}}$ has a left adjoint
 
\[ F_{\law{MON}} \maps \PreNet \to \SMC \] which sends a pre-net
  \[
  \xymatrix{T \ar@<.5ex>[r]^{s} \ar@<-.5ex>[r]_{t} & S^*}
  \]to the strict monoidal category where
  \begin{itemize}
      \item the objects are given by the free monoid $L_{\law{MON}} S$ and,
      \item morphisms are defined inductively as the closure of $T$ under composition $\circ$, and monoidal product $\otimes$. This is quotiented by the axioms
      \begin{itemize}
          \item $(f \otimes g) \otimes h = f \otimes (g \otimes h)$ (the associative law)
          \item $1 \otimes f = f \otimes 1 =f$ (the left and right unit laws) 
          \item $(f_1 \circ g_1) \otimes (f_2 \circ g_2) = (f_1 \otimes f_2) \circ (g_1 \otimes g_2)$ whenever all composites are defined (the interchange law).
      \end{itemize}
  \end{itemize}
  For a morphism of pre-nets $(f,g)$, $F(f,g)$ has object component given by $L_{\law{MON}} g$ and a morphism component given by the unique composition preserving monoid homomorphism extending $f$.
 \end{prop}
 Under the individual token philosophy, monoidal categories are not yet a sufficient semantics for pre-nets. This is because in the individual token philosophy, the order of tokens going in and out of a Petri net must be accounted for. To represent this ordering, we can freely add a swapping morphism
 \[\gamma_{a,b} \maps a \otimes b \to b \otimes a \]
 for every pair of objects $a$ and $b$ in $F_{\law{MON}} (P)$. Every morphism in this new category, can now be regarded as having some permutation of the multisets in its inputs and outputs composed on either side. After choosing an initial ordering on your tokens, these permutations give an order in which tokens flow in and out of each process. Strict monoidal categories equipped with coherent swapping morphisms are called strict symmetric monoidal categories and there is a category where they are objects.
\begin{defn}
  Let $\SSMC$ be the category where 
  \begin{itemize}
      \item objects are strict symmetric monoidal categories and,
      \item morphisms are strict symmetric monoidal functors.
  \end{itemize}
 \end{defn}
 \noindent Symmetries can be freely added to the category $F_{\law{MON}} (P)$ and this process is a left adjoint.
\begin{defn}\label{sym}

 Let \[M \maps \SSMC \to \SMC\]be the forgetful functor which regards every symmetric strict monoidal category as a strict monoidal category and every strict symmetric monoidal functor as a strict monoidal functor. 
 \end{defn}
 
  \begin{prop}\label{MN}
 $M$ has a left adjoint.
 \end{prop}
 
 \begin{proof}
Note that because everything here is strict, the construction of the adjunction between $\SMC$ and $\SSMC$ is a simpler task than constructing an adjunction between their non-strict counterparts. To find a left adjoint to the forgetful functor from symmetric monoidal categories to monoidal categories, the tools of 2-dimensional category theory must be used. On the other hand, due to their strictness, the categories $\SMC$ and $\SSMC$ can be characterized as the category of models for a finite limits theory \cite[\S 3D]{finlim}. To construct these finite limit theories, we start with  $\mathsf{Th(Cat)}$: the well-known finite limit theory for categories. $\mathsf{Th(Cat)}$ has two sorts: $M$ for morphisms and $O$ for objects. The generating operations of $\mathsf{Th(Cat)}$ are
\[\begin{tikzcd} M \ar[r,shift left=1.1ex,"s"] \ar[r,shift right=1.1ex,"t",swap] & 0 \ar[l,"i"description] \end{tikzcd} \quad \circ \maps M \times_O M \to M \]
which represent source, target, identity, and composition. These operations are quotiented to satisfy the operations of a category. To get $\mathsf{Th(SMC)}$, the theory of strict monoidal categories, we add the operations constituting the structure of a monoid on both the objects and the morphisms. These operations have types $M \times M \to M$, $O \times O \to O$, $1 \to M$, and $1\to O$ representing the identity and multiplications of each monoid. These operations are quotiented to satisfy equations expressing the monoid axioms and compatibility with the operations already in $\mathsf{Th(Cat)}$.

$\mathsf{Th(SMC)}$ can be further upgraded to obtain $\mathsf{Th(SSMC)}$, the finite limit theory for strict symmetric monoidal categories. $\mathsf{Th(SSMC)}$ contains all the generating objects and morphisms of $\mathsf{Th(SMC)}$ in addition to a braiding operation with type $O \times O \to M$.  This braiding is required to satisfy the typical axioms of a symmetric monoidal category. Furthermore, there is an inclusion $i \maps \mathsf{Th(SMC)} \to \mathsf{Th(SSMC)}$ which sends every sort and operation of $\mathsf{Th(SMC)}$ to the sort and operation in $\mathsf{Th(SSMC)}$ which plays the same role. In summary, there is a morphism of finite limit theories
 \[ \mathsf{Th(SMC)} \xrightarrow{i} \mathsf{Th(SSMC)}.\]
 Gabriel--Ulmer duality \cite{gu1,gu2} establishes an equivalence between this sort of morphism and filtered colimit preserving right adjoints
 \[\Mod(\mathsf{Th(SSMC)},\Set) \to \Mod(\mathsf{Th(SMC)},\Set).  \]
 This right adjoint is the functor in Definition \ref{sym}. 
 \end{proof}
 
 To get the individual token semantics for pre-nets, first we freely close the transitions under composition and monoidal product using Theorem \ref{big} then we freely add symmetries as shown above.

 \begin{center}
     \begin{tikzcd}
     \PreNet \ar[r,bend left, "F_{\law{MON}}"]  & \ar[l, bend left, "U_{\law{MON}}"] \SMC \ar[r,bend left, "N"] & \SSMC\ar[l,bend left, "M"] 
     \end{tikzcd}
 \end{center}
 This composite adjunction is the adjunction $Z \dashv K$ mentioned in Example \ref{prenet}. This doesn't yet represent the standard individual token semantics for pre-nets. Let 
 \[ f \maps a\otimes a'\to b\]
 be a morphism in $Z(P)$ for a pre-net $P$. Then, composing $f$ with a permutation
 \[ \gamma_{a',a}\maps a' \otimes a\to a \otimes a'\]
 should only represent a new process if $a=a'$. This point could be argued, but the idea is that permuting different places has no functional difference because it requires the same inputs and outputs. Therefore, to get a category which is equivalent to the category of strongly concatenable processes introduced in \cite{SassoneStrong} we must quotient $Z(P)$ by requiring that
 \[ 
 \begin{tikzcd}
 x \ar[r,"\tau"] \ar[d,"a",swap] & y \\
 x \ar[r,"\tau",swap] & y \ar[u,"b",swap]
 \end{tikzcd}
 \]
 commutes for every transition $\tau \maps x \to y$ and permutations $a \maps x \to x$ and $b \maps y \to y$. The key point here is that the permutations are from an element to itself. This only occurs when the permutation switches objects which are the same.
 
 The difference between the adjunction $Z \dashv K$ and the adjunction introduced in \textsl{Functorial Models for Petri Nets} has to do with a change in definition of the morphisms in the category $\PreNet$. We require that the places component come from a function between the sets of places whereas the authors of \cite{functorialsemantics} do not. This change is made so that the category $\PreNet$ admits a smoother description of its semantics.

For practical purposes, it useful to construct semantics for Petri nets rather than pre-nets which have the individual token philosophy. For this we use the functor 
\[\Net{e} \maps \PreNet \to \Petri\] 
introduced in Section \ref{QNet} which sends every pre-net to the Petri net which forgets about the ordering on the input and output of each transtion.

In \cite{functorialsemantics} the authors suggest that an individual token semantics for a Petri net $P$ can computed by choosing a \define{linearization} of $P$; a pre-net $R$ in the preimage $\Net{c}^{-1} (P)$. Then, the semantics category $N \circ F_{\law{MON}} R$ is defined to be the individual token semantics of $P$. Note that this process depends on a choice of linearization only up to isomorphism. Let $R$ and $R'$ be linearizations of $P$, then Theorem 2.5 of \cite{functorialsemantics} proves that $Z (R) \cong Z (R')$.

Choosing a linearization and then applying $Z$ gives an individual token semantics for Petri nets but it is unclear about how to extend this to morphisms of Petri nets. For a functorial construction, we can try to reverse the functor $\Net{c}$ in the  diagram:
\begin{center}
\begin{tikzcd}
\Petri & & \\
\PreNet \arrow[u,"\Net{c}"] \ar[r,"F_{ \law{MON} }"] & \SMC \ar[r,"N"] &  \SSMC
\end{tikzcd}
\end{center}
An inverse to $\Net{c}$ cannot be single valued because there are many linearizations of a given Petri net. For given transition of a Petri net, there are many possible orderings of is source and target. To avoid making a choice, you can make them all. Let $P$ be a Petri net and let $\{(s_i, t_i \maps T \to S^* )\}_{i=1}^n$ be the set of linearizations of $P$, that is the set $\Net{c}^{-1} (P)$. Let $Q$ be the pre-net given by 
\[ 
\begin{tikzcd}
\underset{i=1}{\overset{n}{\Sigma}} T \ar[r,shift left=.5ex,"\Sigma s_i"] \ar[r, shift right=.5ex, "\Sigma {t_i}",swap] & S^*
\end{tikzcd}
\]
where $\Sigma s_i$ and $\Sigma t_i$ denote the copairing of the functions $s_i$ and $t_i$ respectively.
Then, the mapping  
\[ P \mapsto N \circ F_{\law{MON}} \left(Q \right) \]characterizes the category $\mathcal{Q} (P)$ introduced by Sassone in \textsl{On the Category of Petri Net Computations} \cite{SassoneStrong}.
Unfortunately, Sassone showed that this at first only gives a pseudofunctor \cite{SassoneStrong}. This is because there is no obvious way to turn a morphism of Petri nets into a functor between the fibers of their source and target. However Sassone showed that after performing the appropriate quotient on the target category this mapping can be turned into a functor and a left adjoint \cite{SassoneStrong}.

\subsection{Semantics for Integer Nets}
In \textsl{Executions in (Semi-)Integer Petri Nets are Compact Closed Categories}, Genovese and Herold show how compact closed symmetric monoidal categories give a categorical semantics for integer nets \cite{genovese}. Note that these categories are strictly compact closed but not strictly symmetric, i.e., the braidings are not given by the identity. To get an operational semantics for integer nets where the braidings are given by the identity we can use Theorem \ref{big}. This adjunction gives for each integer net a description of its semantics under the collective token philosophy; the morphisms represent the possible executions of an integer net but do not keep track of the identities of the individual tokens.
 \begin{prop}
Let \[\begin{tikzcd} \Set \ar[r,bend left,"L"]\ar[r,phantom,"\bot",pos=.7] & \ar[l,bend left, "R"] \Mod(\law{ABGRP})\end{tikzcd}\]
be the adjunction making up the free abelian group monad $\Z$ in Definition \ref{prenetdef}. Then there is a left adjoint
 \[F_{\law{ABGRP}} \maps \Net{\Z} \to \Mod(\law{ABGRP},\Cat)\]
 which sends a $\Z$-net 
\[
P= \xymatrix{T \ar@<+.5ex>[r]^s \ar@<-.5ex>[r]_t & RLS}
\]
to the $\law{ABGRP}$-category $F_{\law{ABGRP}} (P)$ which has 
\begin{itemize}
    \item $LS$ as its free abelian group of objects and,
    \item morphisms are given by the free closure of T under composition and identity, modulo the axioms of a category, and under monoidal product and inverses, modulo the axioms of an abelian group. We also require that the structure maps are abelian group homomorphisms.
\end{itemize}

For a morphism of $\Z$-nets $(f,g) \maps P \to P'$, $F_{\law{ABGRP}} (f,g)$ is a morphism of $\law{ABGRP}$-categories which is 
\begin{itemize}
    \item given by $Lg$ on objects and,
    \item and morphisms it is given by the unique extension of $f$ which respects the abelian group operation and composition.
\end{itemize}
 \end{prop}
If we wish to construct semantics for integer nets under the individual token philosophy we need our semantics categories to have braidings which are not given by identities. This can accomplished using a similar construction as the previous subsection. Indeed we have a diagram of categories as follows:

\begin{center}
\begin{tikzcd}
\Net{\Z} & & \\
\Net{\law{GRP}} \arrow[u,"\Net{e}"] \ar[r,"F_{ \law{GRP} }" ] & \Mod(\law{GRP},\Cat) \ar[r,"W"] &  \cat{SCCC}
\end{tikzcd}
\end{center}The features of this diagram are as follows.

\begin{itemize}

\item $\Net{e} \maps \Net{\law{GRP}} \to \Net{\Z}$ is abelianization. It sends a $\law{GRP}$-net to the integer net which forgets about the ordering on the input and output of each transition. $\Net{e}$ sends morphisms of $\Net{\law{GRP}}$-nets to themselves.

\item The functor $F_{\law{GRP}} \maps \Net{\law{GRP}} \to \Mod(\Q,\Cat)$ is constructed using Theorem \ref{big}. This functor freely closes the transitions of a $\law{GRP}$-net under the group operation, composition, and freely adds inverses and identites. These semantics categories are required to satisfy the axioms of a group and a category. The structure maps of these are categories are required to be group homomorphisms. This functor will not be explicitly described in this paper but it can be constructed using Theorem \ref{big}.

\item $\cat{SCCC}$ is the category where
\begin{itemize}
    \item An object $C$ is a strictly monoidal strictly compact closed category. This means that for every object $x$ in $C$ and $f$ in $\Mor\, C$ there are inverses with
    \[x \otimes x^{-1} = 1 \text{  and  } f \otimes f^{-1} = 1 \]
    In addition, every pair of objects is equipped with a symmetry $\gamma_{x,y} \maps x \otimes y \to y \otimes x$ satisfying the axioms of a symmetric monoidal category.
    \item Morphisms in $\cat{SCCC}$ are strict symmetric monoidal functors. Preservation of inverses follows from being a strict monoidal functor.
\end{itemize}

\item $W \maps \qCat{\law{GRP}} \to \cat{SCCC}$ is a left adjoint of an adjunction which freely adds a symmetric braiding for every pair of objects. This adjunction is described as follows. The proof of this proposition follows the same argument as Proposition \ref{MN}.
\end{itemize}

\begin{prop}
 Let $X \maps \cat{SCCC} \to \Mod(\law{GRP},\Cat)$ be the forgetful functor which sends objects and morphisms of $\cat{SCCC}$ to their underlying $\law{GRP}$-categories and $\law{GRP}$-functors. Then $X$ has a left adjoint 
 \[W \maps \Mod(\Q,\Cat) \to \cat{SCCC}\]
 which is specified by the following:
 
 \begin{itemize}
     \item for a $\law{GRP}$-category $C$, $WC$ is a symmetric monoidal category such that for every pair of objects $x$, $y$ in $C$ there is an isomorphism $\gamma_{x,y} \maps x \otimes y \to y \otimes x$ satisfiying the axioms of a symmetry in a symmetric monoidal category.
     \item for a $\law{GRP}$-functor $F \maps C \to D$, $WF$ is a the unique extension of $F$ which sends symmetries to symmetries.
 \end{itemize}
\end{prop}

In order to get individual token semantics for an integer net $P$ we can start with a $\law{GRP}$-net $K$ which abelianizes to $P$. This is also called a \define{linearization} of $P$. The individual token semantics of $P$ can be defined as $W \circ F_{\law{GRP}} (K)$. To get a systematic mapping from integer nets to their individual token semantics we can combine all the linearizations of a given integer net. Let $\{(s_i, t_i \maps T_i \to M_{\law{GRP}} (S) )\}_{i=1}^n$ be the set of linearizations of $P$, that is the set $\Net{e}^{-1} (P)$. Let $N$ be the $\law{GRP}$-net given by 
\[ 
\begin{tikzcd}
\underset{i=1}{\overset{n}{\Sigma}} T_i \ar[r,shift left=.5ex,"\Sigma s_i"] \ar[r, shift right=.5ex, "\Sigma {t_i}",swap] & M_{\law{GRP}}(S)
\end{tikzcd}
\]
where $\Sigma s_i$ and $\Sigma t_i$ denote the copairing of the functions $s_i$ and $t_i$ respectively. The mapping 
\[P \mapsto W \circ F_{\law{GRP}} \left( N \right) \]characterizes the category $\mathfrak{F} (P)$ introduced in \cite{genovese}. Like before, this assignment extends to a pseudofunctor rather than a functor. Analogously to the situation for Petri nets, Genovese and Herold prove that after performing some quotients on the target category this turns into a functor and a left adjoint \cite{genovese}.

\subsection{Semantics for Elementary Net Systems}
Theorem \ref{big} can be used to construct a functorial description of the semantics of elementary net systems. This semantics matches the standard description which has not yet been made categorical. For an elementary net system $P$, $F_{\law{SLAT}} (P)$ is a category where the objects are possible markings of $P$ and the morphisms are finite sequences of firings.

 
\begin{prop}
Let 
\[ \begin{tikzcd} \Set \ar[r,phantom,"\bot",pos=.6] \ar[r,bend left,"L"] & \ar[l,bend left,"R"] \Mod(\law{SLAT})\end{tikzcd} \] be the adjunction whose associated monad is $2^{(-)} \maps \Set \to \Set$. 
Then there is a left adjoint
\[F_{\law{SLAT}} \maps \Net{\law{SLAT}} \to \Mod(\law{SLAT},\Cat)\]
which sends an elementary net system
\[ 
\begin{tikzcd}
P =T \ar[r,shift left=.5ex,"s"] \ar[r, shift right=.5ex,"t",swap] & 2^S
\end{tikzcd}
\]
to the $\law{SLAT}$-category $F_{\law{SLAT}} (P)$ with objects given by $2^{S}$ and with morphisms generated inductively by the rules:
\begin{itemize}
    \item for every transition $\tau \in T$, a morphism $\tau \maps s(\tau) \to t(\tau)$ is included,
    \item for every pair of morphisms $f \maps x \to y$ and $g \maps x' \to y'$, their sum $f+g \maps x+x' \to y+y'$ is included,
    \item for every pair of morphisms, $f \maps x \to y$ and $g \maps y \to z$, their composite $g \circ f \maps x \to z$ is included,
    \item these morphisms are quotiented to satisfy the axioms of an idempotent commutative monoid and of a category
    \item these morphisms are quotiented to make composition and the assignment of identities to be monoid homomorphisms.
\end{itemize}
For a morphism of $\law{SLAT}$-nets $(f \maps T \to T', g \maps S \to S')$, 
\[F_{\law{SLAT}} (f,g) \maps F_{\law{SLAT}} (P) \to F_{\law{SLAT}} (P')\]
is the functor given by $2^g$ on objects and by the unique monoidal and functorial extension of $f$ on morphisms.
\end{prop}